\documentclass[a4paper]{amsart}

\pdfoutput=1

\usepackage[utf8]{inputenc}
\usepackage[T1]{fontenc}
\usepackage{lmodern}
\usepackage{amsthm, amssymb, amsmath, amsfonts, mathrsfs}

\usepackage{cancel}

\usepackage{microtype}

\usepackage[pagebackref,colorlinks=true,pdfpagemode=none,urlcolor=blue,linkcolor=blue,citecolor=blue]{hyperref}

\usepackage{relsize}

\usepackage{amsmath,amsfonts,amssymb,amsthm}
\usepackage{amsthm, amssymb, amsmath, amsfonts, mathrsfs}

\usepackage{mathrsfs}
\usepackage{MnSymbol}
\usepackage{scalerel} % for scaling the cube

\usepackage{color}
\usepackage{accents}

\usepackage[normalem]{ulem}
\usepackage{bbm}

\usepackage{cleveref}

\definecolor{labelkey}{gray}{.8}
\definecolor{refkey}{gray}{.8}

\definecolor{darkred}{rgb}{0.9,0.1,0.1}
\definecolor{darkgreen}{rgb}{0,0.5,0}

\newtheorem{theorem}{Theorem}[section]
\newtheorem{lemma}[theorem]{Lemma}

\newtheorem{corollary}[theorem]{Corollary}
\newtheorem{proposition}[theorem]{Proposition}

\theoremstyle{remark}
\newtheorem{remark}[theorem]{Remark}

\renewenvironment{proof}[1][Proof]{ {\itshape \noindent {#1.}} }{$\Box$
\medskip}

\numberwithin{equation}{section}
\newcommand{\cal}{\mathcal}
\newcommand{\R}{\mathbb{R}}
\newcommand{\bbR}{\mathbb{R}}
\newcommand{\bbP}{\mathbb{P}}

\newcommand{\E}{\mathbb{E}}
\newcommand{\bbE}{\mathbb{E}}

\newcommand{\M}{\mathbf{M}}

\newcommand{\cN}{\mathcal{N}}

\newcommand{\eps}{\varepsilon}
\newcommand{\ga}{\gamma}
\def\les{\lesssim}

\newcommand{\la}{\langle}
\newcommand{\ra}{\rangle}

\newcommand{\sw}{\mathsf{w}}
\newcommand{\sU}{\mathsf{U}}
\newcommand{\su}{\mathsf{u}}
\newcommand{\si}{\sigma}
\newcommand{\lam}{\lambda}

\newcommand{\bfn}{\mathbf{n}}
\newcommand{\bfm}{\mathbf{m}}

\newcommand{\be}{\begin{equation}}
\newcommand{\ee}{\end{equation}}

\begin{document}

\title{Gaussian fluctuations from random Schr\"odinger equation}
\author{Yu Gu, Tomasz Komorowski}

\address[Yu Gu]{Department of Mathematics, Carnegie Mellon University, Pittsburgh, PA 15213}

\address[Tomasz Komorowski]{Institute of Mathematics, Polish Academy Of Sciences, ul. \'Sniadeckich 8, 00-636 Warsaw, Poland}

\maketitle

\begin{abstract}
We study the Schr\"odinger equation driven by a weak Brownian forcing, and derive Gaussian fluctuations in the form of a time-inhomogeneous Ornstein-Uhlenbeck process. As a result, when evaluated at a fixed frequency, the intensity of the incoherent wave is of exponential distribution. %as predicted in the physics literature.
\end{abstract}
\maketitle

\section{Introduction}

Wave propagation in random media is a complex phenomenon due to the existence of multiple scales, including the propagation distance, the correlation length and the strength of the media, the wave length, etc. Ideally, one would like to know about the statistical properties of the wave field, in particular its moments information that is needed in practice. This depends in a highly nonlinear way on the statistical properties of the media, and it is usually impossible to resolve all the scales and study the wave equation directly. Thus, in most approaches, effective and simplified models are derived which only involve a few parameters related to the media. There is a large body of literature on the subject and various approximations have been proposed in different asymptotic regimes, see \cite{bkr1,fgps,garnier} and the references therein.

In a high frequency regime, the backscattering of the wave is
  neglected and the forward approximation in a privileged direction
  leads to a Schr\"odinger equation with a random potential. This approach is
used e.g.  to describe the propagation of a
wave beam in a turbulent medium in the forward scattering
approximation of the full wave equation, see \cite{strob}. The
refraction index then plays the role of a potential.

A direct
  study of the random Schr\"odinger equation is still a challenging
  task, with part of the reason being that the moments of the solution do
  not solve closed-form equations, and this makes it hard to extract
  statistical information on the wave field, see e.g. \cite[Chapter 5,6]{bkr1} and
  the references therein.

One can further perform a
  Markovian approximation of the randomness and assume it is
  $\delta-$correlated in the privileged direction. It leads to the
  so-called It\^o-Schr\"odinger model. This model appears e.g. as a diffusive approximation for linear acoustic
waves propagating in $1 + d$ spatial dimensions in a random medium,
when the correlation length of the medium and the  typical
wavelength is much smaller than the propagation distance, see
\cite{GaS}.
Using It\^o calculus one
  can show in this particular case
  that the moments of  the wave function solve closed-form equations.
 %The forward and Markovian approximation has been justified in \cite{}, and the It\^o-Schr\"odinger model is the subject of our paper.

For the It\^o-Schr\"odinger model, the first and second moment equations are straightforward to solve, with the explicit solutions available and corresponding to the ballistic and the scattering component of the wave field respectively. Another important quantity is the fourth moment, as it is related to fluctuations of the intensity of the wave field. The corresponding moment equation is more complicated and can not be solved explicitly. Various approximations were derived from both theoretical and numerical points of view. 

In applications such as light
passing through a turbulent atmosphere or sound waves propagating in
the ocean, it is a well-accepted fact that the distribution of the
complex wave field becomes approximately a complex Gaussian, that is,
the real and imaginary parts are independent Gaussians with the same
variance, and as a consequence, the intensity of the wave field
(given by the square of its absolute value) is of exponential distribution \cite{valley,yaku}. This %is sometimes referred as the scintillation conjecture and 
has been proved in $d=1$ in a randomly layered medium \cite[Chapter 9]{fgps}. Progress has
also been made in high dimensions, focusing on estimating of the fourth
moment to verify the Gaussian summation rule, see \cite{gs1,gs2}.

In the present paper,  we focus on the It\^o-Schr\"odinger model, and
our main contribution is to prove Gaussianity of the wave field in an asymptotic regime
where the medium has a weak strength and the propagation distance is
large. More precisely, we consider the asymptotics of the compensated wave field, see \eqref{comp:wave} below.  This object has been introduced in
\cite{bkr}. It is a field, in both time and momentum variables, that is obtained from
the Fourier transform of the solution of a random Schr\"odinger equation by removing the fast
oscillations of its phase. This is done by  ``recentering''  the
phase through solving the 
free Schr\"{o}dinger equation backward (with no potential). We prove, see
Theorem \ref{t.mainth}, that asymptotically the compensated
wave field converges in law to a complex Gaussian field that is the
solution of an time-inhomogenous Ornstein-Uhlenbeck equation, see
\eqref{e.limiteq} and Section \ref{sec1.2.1} below.

Concerning the method of our proof, we use a martingale representation
for the compensated
wave field, see \eqref{012305-20} and \eqref{e.inteeq}  below. The limit is then verified by
proving the convergence of the respective martingale field appearing in
\eqref{e.inteeq}. This is achieved by  proving the convergence of its
quadratic variation, which constitutes the main thrust of our argument. On the other hand, as the moments of the wave field solve deterministic equations, one can in principle analyze those equations and try to   establish the Gaussian limit in a more analytic way. This perspective has actually been adopted in \cite{gs1,gs2} and many previous works. A message we want to convey here is that, the convergence of the martingale quadratic variation only involves a fourth moment calculation, which simplifies the analysis a bit.

The paper is organized as follows: in Section \ref{sec22} we formulate
our main result.
Some of its aspects are discussed in Section \ref{sec33}. The proof of
the main result is carried out in Sections \ref{sec2}, \ref{s.c} and \ref{sec3}.

\subsection*{Acknowledgements}
We would like to thank two referees for useful comments that help improve the presentation of the paper. Y.G. was partially supported by the NSF through DMS-1907928
and the Center for Nonlinear Analysis of CMU. T.K. acknowledges the support of NCN grant UMO2016/23/B/ST1/00492.

\section{Main result}

\label{sec22}

The equation we study takes the form 
\begin{equation}\label{e.maineq}
\begin{aligned}
&i\partial_t\phi+\frac12\Delta\phi-\eps V(t,x)\circ\phi=0, \quad\quad (t,x)\in\R_+\times\R^d,\\
&\phi(0,x)=\phi_0(x).
\end{aligned}
\end{equation}
Here $\R_+=(0,+\infty)$, the random potential $V(t,x)$ is a real distribution-valued, Gaussian process over some probability space $(\Omega,{\cal V},\bbP)$ that is white in time and smooth in space, with the covariance function 
\[
\E[V(t,x)V(s,y)]=\delta(t-s) R(x-y), \quad (t,x),(s,y)\in\bbR\times\bbR^d.
\]
The parameter $\eps>0$ regulates the strength of the field and
  $\bbE$ is the expectation with respect to $\bbP$. We assume that $R(\cdot)$ belongs to the Schwarz class $\mathcal{S}(\R^d)$.
The
product $\circ$ between the solution and the noise  is in the
Stratonovitch sense. Equation \eqref{e.maineq} is understood via the
corresponding It\^o stochastic partial differential equation
\begin{equation}\label{e.eqito}
d\phi=\left(\frac i2\Delta\phi-\frac{\eps^2}{2}R(0)\phi\right)dt-i\eps B(dt)\phi,
\end{equation}
where $\left(B(t)\right)_{t\ge0}$ is a smooth in space Wiener
process such that $\dot B(t,x)=V(t,x)$, i.e. it is a Gaussian random field with the covariance function
\begin{equation}
\label{B-cor}
\E[B(t,x)B(s,y)]=(t\wedge s) R(x-y),\quad (t,x),(s,y)\in\bar\bbR_+\times\bbR^d. 
\end{equation}
The equation was analyzed in the early work \cite{dg}, and it was shown that the $L^2$ norm of the wave
  function 
is conserved \cite[Equation (2.19)]{dg}, i.e.
$$
\|\phi(t,\cdot)\|_{L^2(\bbR^d)}=\|\phi_0 \|_{L^2(\bbR^d)},\quad
t\ge0,\quad \bbP\quad\mbox{a.s.}
$$
We shall denote by $\|\cdot\|_{L^p(\bbR^d)}$ the $L^p$ norm with
respect to the Lebesgue measure over $\bbR^d$.
The Fourier transform of a given function $f\in L^2(\bbR^d)$ shall be denoted by
$\hat{f}(\xi):=\int_{\bbR^d} f(x)e^{-i\xi\cdot x}dx$, $\xi\in\bbR^d$.

Before stating our main result, let us introduce some definitions.
Let $\sw (t,\cdot,\xi)$ be the finite measure-valued
solution of  the linear kinetic equation
\begin{equation}
\label{020405-20}
\begin{aligned}
&\partial_t \sw(t,\cdot,\xi)+\xi\cdot \nabla_x \sw(t,\cdot,\xi)=\int_{\R^d}\frac{\hat{R}(p)}{(2\pi)^d} [\sw(t, \cdot,\xi-p)-\sw(t, \cdot,\xi)]dp,\\
&\sw(0,dx,\xi)=|\hat{\phi}_0(\xi)|^2 \delta_0(dx).
\end{aligned}
\end{equation}
Here $\delta_0$ is the Dirac measure at $0$. Define the measure
\begin{equation}
\label{010405-20a}
\su(t,dx,\xi):=\int_{\R^d}\sw(t,dx,\xi-p) \frac{\hat{R}(p) dp}{(2\pi)^d} .
\end{equation}
Assume throughout the paper that $\hat{\phi}_0\in C_b(\R^d)\cap
L^2(\bbR^d)\cap \mathrm{Lip}(\R^d)$, where  $C_b(\R^d)$ and
$\mathrm{Lip}(\R^d)$ denote the spaces of bounded and continuous, and
Lipschitz continuous functions, respectively. The following simple fact holds.
\begin{proposition}
\label{prop010605-20}
For each $(t,\xi)\in  \bbR_+\times \bbR^d$
the measure $\su(t,\cdot,\xi)$ is absolutely continuous  with respect to the Lebesgue
measure.  Its density $\sU(t,\cdot,\xi)$ is 
 strictly positive and smooth.
\end{proposition}
The proof of the proposition is presented in Section \ref{sec2}.

Next we define the {\em compensated wave field}
\begin{equation}
\label{comp:wave}
X_\xi^{\eps}(t,\eta):=\hat{\phi}\left(\frac{t}{\eps^2},\xi+\eps^2
\eta\right)\exp\left\{\frac{it}{2\eps^2}|\xi+\eps^2\eta|^{2}\right\},\quad (t,\eta)\in \bar\bbR_+\times\bbR^d,
\end{equation}
see Remark \ref{rm1.3} below for a discussion on the interpretation of the
field.

Our main results concerns the long time, large scale behavior of the
Fourier transform $\hat\phi(t,\xi)$ of the wave function and can be
stated as follows. 
\begin{theorem}\label{t.mainth}
Fix
$\xi\in\R^d$. The following convergence holds:
\[
\left\{X_\xi^{\eps}(t,\eta)\right\}_{(t,\eta)\in\bar\bbR_+\times\bbR^d}
\Rightarrow
\{X_\xi(t,\eta)\}_{(t,\eta)\in\bar\bbR_+\times\bbR^d},\quad \mbox{as }\eps\to0,
\]  
in law  over $C(\bar\bbR_+\times \R^d)$. The limit $X_\xi$ is a
complex valued Gaussian process admitting the representation
\begin{equation}\label{e.limiteq}
X_\xi(t,\eta)=\hat{\phi}_0(\xi)e^{-\frac12R(0)t}+\int_0^t e^{-\frac12R(0)(t-s)}B_\xi(ds,\eta),
%dX_\xi(t,\eta)=-\tfrac12R(0)X_\xi(t,\eta)dt+dB_\xi(t,\eta),\quad\quad X_\xi(0,\eta)=\hat{\phi}_0(\xi),
\end{equation}
where $B_\xi$ is a zero mean complex Gaussian process with the covariance function 
\begin{equation}
\label{010405-20}
\begin{aligned}
&\E[B_\xi(t_1,\eta_1)B_\xi^*(t_2,\eta_2)]=\int_0^{t_1\wedge t_2}\hat{\sU}(s,\eta_1-\eta_2,\xi)ds,\\
&\E[B_\xi(t_1,\eta_1)B_\xi(t_2,\eta_2)]=0,\quad (t_j,\eta_j)\in\bar\bbR_+\times\bbR^d,\,j=1,2.
\end{aligned}
\end{equation}
The function $(t,\eta)\mapsto \hat{\sU}(t,\eta,\xi)$,
$(t,\eta)\in\bar\bbR_+\times\bbR^d$ is
given by
\begin{equation}
\label{010405-20c}
\hat{\sU}(t,\eta,\xi):=\int_{\R^d}e^{-i\eta\cdot y}\sU(t,y+\xi t,\xi)dy.
\end{equation}
\end{theorem}
The proof of the theorem is presented in Section \ref{sec3}.

%The solution $X_\xi$ is written explicitly as 
%\[
%X_\xi(t,\eta)=\hat{\phi}_0(\xi)e^{-\frac12R(0)t}+\int_0^t e^{-\frac12R(0)(t-s)}dB_\xi(s,\eta)
%\]
%We emphasize that the above stochastic integral is in the It\^o sense and is only with respect to the time variable with $\eta$ treated as a parameter. 
%To compute the covariance, we first have 
%\[
%\begin{aligned}
%\E[X_\xi(t,\eta_1)X_\xi^*(t,\eta_2)]=&|\hat{\phi}_0(\xi)|^2e^{-R(0)t}\\
%&+\int_0^t e^{-R(0)(t-s)}\int_{\R^d}\frac{\hat{R}(p)}{(2\pi)^d}e^{ip\cdot(\eta_1-\eta_2)s}\widehat{\sW}(s,\eta_1-\eta_2,\xi-p)dpds
%\end{aligned}
%\]
%If we take $\eta_1=\eta_2=\eta$, then 
%\[
%\E[|X_\xi(t,\eta)|^2]=|\hat{\phi}_0(\xi)|^2e^{-R(0)t}+\int_0^t e^{-R(0)(t-s)}\int_{\R^d}\frac{\hat{R}(p)}{(2\pi)^d}\widehat{\sW}(s,0,\xi-p)dpds
%\]

\section{Discussion}

\label{sec33}

\subsection{On the interpretation of the result}

\label{sec011109-20}

\subsubsection{On the initial data}

Suppose that the initial data varies on
the microscopic scale and is described  by the family of wave functions
\begin{equation}
\label{011109-20}
\tilde\chi_\eps(y):=\frac{1}{\eps^{d}}\phi_0\left( y\right),\quad y\in\bbR^d,
\end{equation}
where $y$ is the spatial coordinate in the microscopic units. We
assume that the macroscopic coordinate is  given by
$x=\eps^2 y$, so the prefactor $\eps^{-d}$ in the left hand side of
\eqref{011109-20} assures that the macroscopic energy density of the
wave is of order $O(1)$, provided that $\phi_0\in L^2(\bbR^d)$. The initial data is
 fast oscillating on the macroscopic scale and  is described by the
 initial profile
\begin{equation}
\label{011109-20a}
\tilde\phi_\eps(x):=\frac{1}{\eps^{d}}\phi_0\left(\frac{x}{\eps^2}\right),\quad x\in\bbR^d,
\end{equation}
with $\phi_0\in L^2(\bbR^d)$. The family $(\tilde\phi_\eps)_{\eps\in(0,1]}$ forms a bounded set in
$L^2(\bbR^d)$.

\subsubsection{Compensated wave-function, Wigner and smoothed Wigner
  functions}

Consider now
$\tilde\phi_\eps(t,x)=\eps^{-d}\phi\left(\frac{t}{\eps^2},
  \frac{x}{\eps^2}\right)$, where $\phi$ is the solution  of
\eqref{e.maineq}. Since the laws of the noise $\frac{1}{\eps}
V\left(\frac{t}{\eps^2},\frac{x}{\eps^2}\right)$ and that of $
V\left(t,\frac{x}{\eps^2}\right)$ are identical,  the  law of $\tilde\phi_\eps(t,x)$
coincides with   that of the solution of the equation
\begin{equation}\label{e.maineq-eps}
\begin{aligned}
&i\partial_t \tilde\phi_\eps(t,x)+\frac{\eps^2}{2}\Delta \tilde\phi_\eps(t,x)-V\left(t,\frac{x}{\eps^2}\right)\circ \tilde\phi_\eps(t,x)=0, \quad\quad (t,x)\in\R_+\times\R^d,\\
&\tilde\phi_\eps(0,x)=\tilde\phi_\eps(x).
\end{aligned}
\end{equation}
Furthermore, the Fourier transform is related to the unscaled wave function through
\begin{equation}
\label{031109-20}
\hat{\tilde\phi}_\eps\left(t,\frac{\xi}{\eps^2}\right)=\eps^d\hat{\phi}\left(\frac{t}{\eps^2},\xi\right).
\end{equation}
The compensated wave-function, defined as  
\begin{equation}
\label{comp:wave1}
X_\xi^{\eps}(t,\eta):=\frac{1}{\eps^{d}}\hat{\tilde\phi}_\eps\left(t,\frac{\xi +\eps^2\eta}{\eps^2}
\right)\exp\left\{\frac{it}{2\eps^2}|\xi+\eps^2\eta|^{2}\right\},\quad (t,\eta)\in \bar\bbR_+\times\bbR^d.
\end{equation}
is given by the expression \eqref{comp:wave}.
Its inverse Fourier transform is a spectral measure defined by the equality
$$
\int_{\bbR^d}{\cal
  X}_\eps(t,dx,\xi)J^*(x):=\frac{1}{(2\pi)^d}\int_{\bbR^d}X_\xi^{\eps}(t,\eta)\hat
J^*(\eta)d\eta
$$
for any $J$ that belongs to the Schwartz class ${\cal S}(\bbR^d)$. For
fixed $\eps>0,\xi\in\R^d,t>0$,  $X^\eps_\xi(t,\cdot)$ belongs to
$L^2(\R^d)$ almost surely. Therefore, we know that ${\cal
  X}_\eps(t,dx,\xi)$ actually has a density in $x$, which we shall also
denote, with some abuse of the notation,  by ${\cal X}_\eps(t,x,\xi)$.
Denote  by
$
{\cal X}(t,dx,\xi)$ the respective spectral measure associated with the
stationary field $\eta\mapsto X_\xi(t,\eta)$.

Note that $X_\xi^\eps(0,\eta)=\hat{\phi}_0(\xi+\eps^2\eta)$, so obviously we have
\begin{equation}
\label{021109-20aa}
\lim_{\eps\to0}
\int_{\bbR^{2d}}{\cal X}_\eps(0,x,\xi)J^*(x,\xi)dxd\xi
=\int_{\bbR^{d}}\hat\phi_0(\xi)\hat{J}^*(0,\xi)d\xi
\end{equation}
for any test function $J\in{\cal
  S}(\bbR^{2d}).$ 
Hence
\begin{equation}
\label{021109-20aa}
\lim_{\eps\to0}
{\cal X}_\eps(0,x,\xi)
=\hat\phi_0(\xi)\delta(x),
\end{equation}
$*$-weakly in ${\cal
  S}'(\bbR^{2d}).$
The following result holds.
\begin{proposition}\label{c.c1a}
Fix any test function $J\in {\cal S}(\bbR^{2d})$
and  $t\ge 0$. Then, for any $\xi\in\bbR^d$ we have 
\begin{equation}
\label{021109-20ab}
\lim_{\eps\to0}\int_{\bbR^{d}}{\cal X}_\eps(t,x,\xi)J^*(x,\xi)dx
=\int_{\bbR^{d}}{\cal X}(t,dx,\xi)J^*(x,\xi)
\end{equation}
in law. In addition
\begin{equation}
\label{021109-20a}
\begin{split}
&
\lim_{\eps\to0}\int_{\bbR^{2d}}{\cal X}_\eps(t,x,\xi)J^*(x,\xi)dxd\xi
=\frac{e^{-\frac12tR(0)}}{(2\pi)^d}\int_{\bbR^{2d}}
\hat{J}^*(\eta,\xi)\hat\phi_0(\xi)   d\eta d \xi\\
&
=e^{-\frac12tR(0)}\int_{\bbR^{d}}
J^*(0,\xi)\hat\phi_0(\xi)   d\xi
\end{split}
\end{equation}
in $L^2(\Omega)$, as $\eps\to0$.%the $L^2$ sense over the probability space, as $\eps\to0$.
\end{proposition}
The proof of the proposition is contained in Section \ref{s.c}.

%\$$$$$

% Let us mention in passing that in \cite{gs2} the case of slowly
% varying microscopic initial data has been considered. The respective microscopic and macroscopic profiles for the initial data considered there
% are
% described by the functions  
% \begin{equation}
% \label{011109-20b}
% \tilde\chi_\eps'(y):=\phi_0\left(\eps^2 y\right),\quad\mbox{and}\quad \tilde\phi_\eps'(x):=\phi_0\left(x\right)
% \end{equation}
% where $y$ and $x$ correspond to the microscopic and macrscopic
% coordinates respectively.

It is worthwhile to compare the behavior of the compensated
wave function  with that of the Wigner functions corresponding to the family $(\tilde\phi_\eps)_{\eps\in(0,1]}$, cf
e.g. \cite[Item 1), p. 557]{lions-paul},
\begin{equation}
\label{Wigner-z}
\sw_\eps(t,x,\xi):=\int_{\bbR^d}\tilde\phi_\eps\left(t,x+\frac{\eps^2 y}{2}\right)
\tilde\phi_\eps^*\left(t,x-\frac{\eps^2  y}{2}\right)e^{-i\xi\cdot
  y}dy.
\end{equation}
By taking the Fourier transform, we obtain
\begin{equation}
\label{Wigner-y}
\begin{split}
&
\sw_\eps(t,x,\xi)
=\frac{1}{(2\pi \eps^2)^{d}}\int_{\bbR^d}\hat{\tilde\phi}_\eps\left(t,\frac{\xi}{\eps^2}+\frac{\eta}{2}\right)
\hat{\tilde\phi}_\eps^*\left(t,\frac{\xi}{\eps^2}-\frac{\eta}{2}\right)e^{i\eta\cdot
  x}
d\eta\\
&
=\frac{1}{(2\pi)^{d}}\int_{\R^d} \hat{\phi}\left(\frac{t}{\eps^2},\xi+\frac{\eps^2\eta}{2}\right)\hat{\phi}^*\left(\frac{t}{\eps^2},\xi-\frac{\eps^2\eta}{2}\right)e^{i\eta\cdot x} d\eta. 
\end{split}
\end{equation}
Using \eqref{comp:wave} we get
\begin{equation}\label{e.defwigner}
\sw_\eps(t,x,\xi)
=\frac{1}{(2\pi)^{d}}\int_{\R^d} X_\xi^\eps \left (t,\frac{\eta}{2}\right)(X_\xi^{\eps})^*\left (t,-\frac{\eta}{2}\right)e^{i\eta \cdot(x-\xi t)}d\eta.
\end{equation}

% It is natural to consider the wave function expressed in the physical variable, i.e., the function $\phi(t,x)$, and also the Wigner transform  
% which, in the physical domain, can be written as 
% \begin{equation}\label{e.9101}
% \sw_\eps(t,x,\xi)= \left(\frac{2\pi}{\eps^2}\right)^d\int_{\R^d} \phi(\frac{t}{\eps^2},\frac{x}{\eps^2}-\frac{y}{2})\phi^*(\frac{t}{\eps^2},\frac{x}{\eps^2}+\frac{y}{2})e^{i\xi \cdot y}dy.
% \end{equation}
Theorem~\ref{t.mainth} implies the following, see Section~\ref{s.c} below for the proof.
\begin{proposition}\label{c.c1}
Fix any test function $J\in {\cal S}(\bbR^{2d})$  
and $t>0$. Then, for any $\xi\in\R^d$ we have 
\begin{equation}\label{e.982}
\int_{\R^d} \sw_\eps(t,x,\xi)J^*(x,\xi)dx \Rightarrow \frac{1}{(2\pi)^{d}}\int_{\R^d}
X_\xi\left(t,-\frac{\eta}{2}\right)X_\xi^*\left(t,\frac{\eta}{2}\right)e^{-i\eta \cdot\xi
  t}\hat J^*(\eta,\xi)d\eta,
\end{equation}
in distribution, as $\eps\to0$. In addition, (cf \eqref{011609-20} below)
\begin{equation}\label{e.982}
\begin{aligned}
\lim_{\eps\to0}\int_{\R^{2d}} \sw_\eps(t,x,\xi)J^*(x,\xi)dx d\xi =&\frac{1}{(2\pi)^{d}}\int_{\R^d}
\E\left[X_\xi\left(t,-\frac{\eta}{2}\right)X_\xi^*\left(t,\frac{\eta}{2}\right)\right]e^{-i\eta \cdot\xi
  t}\hat J^*(\eta,\xi)d\eta d\xi\\
=&\int_{\R^{2d}}
\sw(t,x,\xi) J^*(x,\xi)dx d\xi,
\end{aligned}
\end{equation}
in $L^2(\Omega)$, as $\eps\to0$.
\end{proposition}
%\begin{remark}
% The factor $(2\pi)^{-d}$ on the r.h.s. of \eqref{e.982} comes from the fact that $\sw_\eps(0,x,\xi)$ converges $*$-weakly in ${\cal
%   S}'(\bbR^{2d})$ to $(2\pi)^{-d}|\hat{\phi}_0(\xi)|^2\delta(x)$ which equals to $(2\pi)^{-d}\sw(0,x,\xi)$ with $\sw$ defined in \eqref{020405-20}.
% \end{remark}

Formula \eqref{021109-20aa} shows that our highly oscillatory initial data
localizes at $x=0$. To ``smear'' the observed position around $0$
we introduce also the {\em smoothed Wigner function}, cf  \cite[formula (89), p. 65]{gs1}, defined as follows
\begin{equation}
\label{Wigner-sf}
\sw_\eps^{\mathrm{s}}(t,x,\xi)
:=\frac{1}{2^{d/2}\pi^{3d/2}\eps^{2d}}\left|\int_{\bbR^d}e^{i\eta\cdot
  x}e^{-|\eta|^2}\hat{\tilde \phi}_\eps^*\left(t,\frac{\xi}{\eps^2}-\eta\right)d\eta\right|^2.
\end{equation}
A simple calculation shows that
\begin{equation}
\label{Wigner-sfc}
  \sw_\eps^{\mathrm{s}}(t,x,\xi)
=\frac{1}{(2\pi)^d}\int_{\bbR^{2d}}\sw_\eps\left(t,x-y,\xi-\frac{\eps^2\eta}{2}\right)
  \exp\left\{-\frac{|y|^2}{2}-\frac{|\eta|^2}{2}\right\}dy d\eta.
\end{equation}
In other words, $\sw_\eps^{\mathrm{s}}$ is an average of $\sw_\eps$ on the $O(1)-$scale of the spatial variable and  on the $O(\eps^2)-$scale of the frequency variable. Using \eqref{e.defwigner} and \eqref{Wigner-sf}, for the family $(\tilde\phi_\eps)_{\eps\in(0,1]}$ given by
\eqref{031109-20}, we get therefore
\begin{equation}
\label{Wigner-sf1}
\begin{split}
&\sw_\eps^{\mathrm{s}}(t,x,\xi)=\frac{1}{2^{d/2}\pi^{3d/2}}\int_{\bbR^{2d}}(X_{\xi}^{\eps})^*\left
  (t,-\eta\right) X_{\xi}^\eps \left
  (t,-\eta'\right)\\
&
\times e^{i(\eta-\eta') \cdot(x-\xi t)}
  \exp\left\{-|\eta|^2\left(1-\frac{i\eps^2t}{2}\right)-|\eta'|^2\left(1+\frac{i\eps^2t}{2}\right)\right\} d\eta d\eta'.
\end{split}
\end{equation}
From Theorem~\ref{t.mainth}, we conclude immediately the following.
\begin{proposition}\label{prop011409-20}
Fix any $(t,x,\xi)\in\bbR_+\times\R^{2d}$. Then,
\begin{equation}\label{e.982-1}
  \sw_\eps^{\mathrm{s}}(t,x,\xi)\Rightarrow \frac{1}{2^{d/2}\pi^{3d/2}}\left|\int_{\bbR^d}e^{i\eta\cdot
  (x-\xi t)}e^{-|\eta|^2}X_{\xi}^*\left
  (t,-\eta\right) d\eta\right|^2,
\end{equation}
in distribution as $\eps\to0$.
\end{proposition}

\subsection{$X_\xi(t,\eta)$ as an inhomogeneous Ornstein-Uhlenbeck
  process}
\label{sec1.2.1}

Let us make a few comments about the limiting equation
\eqref{e.limiteq}.
 Note that in light of Proposition \ref{prop010605-20}, we can write
\begin{equation}
\label{cyl}
B_\xi(t,\eta)=\int_0^t \int_{\R^d}e^{-i\eta\cdot y}\sU^{1/2}(s,y+\xi
s,\xi)B_{\rm w}(ds,dy),
\end{equation}
where $B_{\rm w} (ds,dy)$ is complex valued space-time  white noise,
i.e.
\begin{equation}
\label{010405-20z}
\begin{aligned}
&\E[B_{\rm w} (dt,dx) B_{\rm w}^*(ds,dy)]=\delta(t-s)\delta(x-y)dtdsdxdy,\\
&\E[B_{\rm w} (dt,dx) B_{\rm w} (ds,dy)] =0,\quad (t,x),\, (s,y)\in\bbR\times\bbR^d.
\end{aligned}
\end{equation}
We can write therefore
\begin{equation}
\label{010605-20}
\left\{
\begin{array}{l}
dX_\xi(t,\eta)=-\frac12R(0)X_\xi(t,\eta)dt+\mathlarger{\int}_{\R^d}e^{-i\eta\cdot y}\sU^{1/2}(t,y+\xi
t,\xi) B_{\rm w} (dt,dy), \\
X_\xi(0,\eta)=\hat{\phi}_0(\xi),
\end{array}
\right.
\end{equation}
so for fixed $\xi,\eta\in\R^d$, $X_\xi(t,\eta)$ is actually a time-inhomogeneous  Ornstein-Uhlenbeck
process. 

\subsection{The covariance structure of $X_\xi(t,\eta)$}
For fixed $(t,\xi)\in\R_+\times\R^d$, the field $\eta\mapsto
X_\xi(t,\eta)$, $\eta\in\bbR^d$ is a stationary complex-valued
Gaussian. A direct calculation, using \eqref{010605-20} and
\eqref{020405-20}, shows that its second absolute moment
$\tilde{\sw}(t,\xi):=\E[|X_\xi(t,\eta)|^2]$ satisfies the homogeneous
linear Boltzmann equation 
\begin{equation}\label{e.eqmomentum}
\partial_t \tilde{\sw}(t,\xi)=\int_{\R^d}\frac{\hat{R}(p)}{(2\pi)^d} [\tilde{\sw}(t,\xi-p)-\tilde{\sw}(t,\xi)]dp, \quad\quad \tilde{\sw}(0,\xi)=|\hat{\phi}_0(\xi)|^2.
\end{equation}
We can also compute its mean and the covariance function:
\begin{equation}
\begin{aligned}\label{e.cov1}
&
\bbE X_\xi(t,\eta)=\hat{\phi}_0(\xi)e^{-\frac12R(0)t},\\
&
{\rm Cov}\big(X_\xi(t,\eta_1),X_\xi(t,\eta_2)\big)=\int_0^t\int_{\R^d}e^{-R(0)(t-s)}\hat{\sU}(s,\eta_1-\eta_2,\xi) ds.
\end{aligned}
\end{equation}
From the equation satisfied by $\sw$, see \eqref{020405-20}, we conclude that $\hat{\sw}$
- its  Fourier transform in the $x$ variable - satisfies the equation:
\begin{equation}
\label{030605-20}
\begin{aligned}
&\partial_t \hat{\sw}(t,\eta,\xi)+i\xi\cdot \eta\, \hat{\sw}(t,\eta,\xi)=\int_{\R^d}\frac{\hat{R}(p)}{(2\pi)^d} \hat{\sw}(t,\eta,\xi-p)dp-R(0)\hat{\sw}(t,\eta,\xi),\\
&\hat{\sw}(t,\eta,\xi)=|\hat{\phi}_0(\xi)|^2.
\end{aligned}
\end{equation}

Define 
\begin{equation}
\label{030605-20}
\hat{\mathsf{f}}(t,\eta,\xi):=\hat{\sw}(t,\eta,\xi)e^{i\xi\cdot
  \eta t}.
\end{equation}
From \eqref{030605-20} we conclude that it satisfies 
the integral equation
\begin{equation}\label{e.cov2}
\hat{\mathsf{f}}(t,\eta,\xi)=|\hat{\phi}_0(\xi)|^2e^{-R(0)t}+\int_0^te^{-R(0)(t-s)}\int_{\R^d} \frac{\hat{R}(p)}{(2\pi)^d}e^{ip\cdot \eta s}\hat{\mathsf{f}}(s,\eta,\xi-p)dpds.
\end{equation}
The solution can be written as an infinite series expansion
\begin{equation}\label{e.cov3}
\begin{aligned}
&\hat{\mathsf{f}}
(t,\eta,\xi) =e^{-R(0)t}\left\{|\hat{\phi}_0(\xi)|^2\vphantom{\int_0^1}\right.\\
&
\left.+\sum_{n\geq
    1}\int_{[0,t]_<^n}\int_{\R^{nd}} \prod_{j=1}^n
  \frac{\hat{R}(p_j)e^{i\eta\cdot p_j
      s_j}}{(2\pi)^d}|\hat{\phi}_0(\xi-p_1-\ldots-p_n)|^2d\mathbf{p}_{1,n}d\mathbf{s}_{1,n}\right\}.
\end{aligned}
\end{equation}
Here 
\[[0,t]_<^n=\{0<s_n<\ldots<s_1<t\}
\] is the $n$ dimensional
simplex, $d\mathbf{p}_{1,n}:=dp_1\ldots dp_n$, $d{\bf s}_{1,n}=ds_1\ldots
ds_n$.

Comparing \eqref{030605-20} with  \eqref{010405-20c}, we obtain
\begin{equation}
\label{020605-20}
\int_{\R^d} \frac{\hat{R}(p)}{(2\pi)^d}e^{ip\cdot \eta s}\hat{\mathsf{f}}(s,\eta,\xi-p)dp= \hat{\sU}(s,\eta,\xi),
\end{equation}
therefore, from \eqref{e.cov1} and \eqref{e.cov2}, we conclude that 
\begin{equation}
\label{011609-20}
\E[X_\xi(t,\eta_1)X_\xi^*(t,\eta_2)]=\hat{\mathsf{f}} (t,\eta_1-\eta_2,\xi)=\hat{\sw}(t,\eta_1-\eta_2,\xi)e^{i\xi\cdot
  (\eta_1-\eta_2) t}.
\end{equation}
\medskip

\subsection{Some additional remarks}
%We make a few remarks on our result. 

\begin{remark}
The limit of $X_\xi^\eps(t,\eta)$
can be written as 
 $$
X_\xi(t,\eta)=\bbE X_\xi(t,\eta)+\tilde X_\xi(t,\eta).
 $$
 The mean represents the ballistic component of the
 (compensated) wave field and the fluctuating part  corresponds to the
 scattering (random) component and is given by a stochastic
 convolution, see formula \eqref{e.limiteq}. Since the latter term is a complex Gaussian,
 i.e., its real and imaginary part are independent zero mean Gaussians
 with the same variance, given by
\[
\sigma^2(t,\xi)=\frac12 \bbE |\tilde X_\xi^\eps(t,\eta)|^2=\frac12\bigg(\tilde{\sw}(t,\xi)-|\hat{\phi}_0(\xi)|^2e^{-R(0)t}\bigg),
\]
where $\tilde{\sw}$ solves \eqref{e.eqmomentum}. It is an elementary
fact that the intensity of the scattering component, defined as
  $|\tilde X_\xi^\eps(t,\eta)|^2$, is of exponential distribution $\mathrm{Exp}(\tfrac{1}{2\sigma^2(t,\xi)})$.
\end{remark}

\begin{remark}
\label{rm1.3}
Theorem \ref{t.mainth} concerns the asymptotics of  the compensated
wave function
$X_\xi^\eps(t,\eta)$, defined in \eqref{comp:wave}. As can be seen
from its definition, the field is obtained from
the Fourier transform of the solution  of \eqref{e.maineq} by ``compensating'' with the fast oscillating phase corresponding to the
free Schr\"{o}dinger equation (with no potential present). This object has been
introduced in \cite{bkr}.
To prove Gaussianity of the scaled limit, we 
use the white-in-time structure of the potential and study the relevant martingales.
By invoking a martingale central limit theorem, the argument reduces
to proving the convergence of quadratic variations.  This in turn
involves a fourth moment calculation and  we deal with it via
a diagram expansion method. The $\delta-$correlation in time simplifies the diagrams significantly,
compared to the smoothly correlated case. On a heuristic level, the
convergence of the quadratic variation to a deterministic limit comes
from a self-averaging effect, which roughly says that for any two
distinct frequencies $\xi_1\neq \xi_2$ the wave function evaluated at
$\xi_1$ and $\xi_2$ are asymptotically independent hence the
randomness disappears after the averaging. We will discuss this
phenomenon in more details in Remark~\ref{r.heuristics}.
\end{remark}

\begin{remark}
Although we do not provide the details, analogous results for other
dispersion relations can be derived almost verbatim. More precisely we
can replace the laplacian $
\Delta$ by its fractional counterpart $-|\Delta|^\alpha$ for any
$\alpha>0$, or operators defined by other Fourier multipliers, and
prove a similar result, with an appropriately adjusted scaling. An interesting feature is that the dispersion relation neither affects
the limiting Gaussian nature nor the marginal distribution. As it will
become clear later on, the Gaussianity comes from the magnitude of the phase, rather than its specific structure, and the second moment calculation in \eqref{e.eqmomentum} does not involve the phase information due to the unitary evolution of the Schr\"odinger equation, so the marginal distribution is always the same. The dispersion relation only shows up in the limiting covariance structure. 
\end{remark}

%\textcolor{red}{\bf DOTAD}

\begin{remark}
\label{rmk1.5}
The rescaled wave function $\phi_\eps(t,x)=\phi(\frac{t}{\eps^2},x)$ satisfies
\begin{equation}
\label{010705-20}
i\partial_t\phi_\eps+\frac{1}{2\eps^2}\Delta\phi_\eps-\frac{1}{\eps}V(\frac{t}{\eps^2},x)\circ\phi_\eps=0.
\end{equation}
Since the potential is $\delta-$correlated in time, the laws of the
fields
$\left(\tfrac{1}{\eps}V(\tfrac{t}{\eps^2},x)\right)_{(t,x)\in\bbR^{1+d}}$
and that of  $\left( V(t,x) \right)_{(t,x)\in\bbR^{1+d}}$
coincide. Therefore, the law of $\left( \phi_\eps(t,x)
\right)_{(t,x)\in\bbR^{1+d}}$ is the same as that of  the solution of 
\begin{equation}\label{e.531}
i\partial_t\phi_\eps+\frac{1}{2\eps^2}\Delta\phi_\eps-V(t,x)\circ\phi_\eps=0.
\end{equation}
It is clear that we can replace the factor $\tfrac{1}{\eps^2}$ in the above
equation by $\tfrac{1}{\ga}$, send $\ga\to0$, and adjust
accordingly the
formulation of our result.
\end{remark}

\begin{remark}
From an application point of view, we start with the Schr\"odinger equation of the form 
\[
i\partial_t \Phi+\frac{c_0}{2\omega}\Delta \Phi-\frac{\omega\sigma}{2c_0}V(t,x)\circ \Phi=0,
\] which is the standard form that comes from the paraxial approximation, see e.g. \cite{prs}. Here $c_0\sim O(1)$ is a constant of  describing the average wave speed, $\omega\gg1$ is the wave frequency, and $\sigma$ is the strength of random media. Now suppose we consider a propagation distance of order $L\gg1$, so $\Phi_L(t,x)=\Phi(tL,x)$ satisfies 
\[
i\partial_t \Phi_L+\frac{c_0L}{2\omega}\Delta \Phi-\frac{L\omega\sigma}{2c_0}V(tL,x)\circ\Phi_L=0.
\]
By the scaling property of $V$, $\Phi_L$ has the same law as the solution to 
\[
i\partial_t \Phi_L+\frac{c_0L}{2\omega}\Delta \Phi-\frac{\sqrt{L}\omega\sigma}{2c_0}V(t,x)\circ \Phi_L=0.
\]
Thus, if the parameters $(\omega,\sigma,L)$ satisfy 
\[
\sqrt{L}\omega\sigma\sim O(1), \quad\mbox{and }\quad \frac{L}{\omega}\gg1,
\] we are in the regime given by \eqref{e.531} and the result in the
paper shows that the compensated wave function has then an approximate Gaussian distribution.
\end{remark}

\begin{remark}
Another interesting scaling regime one can consider is the following. Suppose we start with a Schrodinger equation with a random driving force of order $O(1)$:
\[
i\partial_t\phi+\frac12\Delta\phi- V(t,x)\circ\phi=0.
\]
The goal is to study the long time behavior of $\phi(t,x)$. It can be
checked that the second moment $\E[|\hat{\phi}(t,\xi)|^2]$ equals to
$\tilde{\sw}(t,\xi)$, which evolves according to
\eqref{e.eqmomentum}. From the probabilistic point of view, the
equation is associated with a Markov jump process  corresponding to
the momentum variable. It performs a jump process with the kernel $(2\pi)^{-d}\hat{R}(p)$. A standard diffusion approximation yields that $\tilde{\sw}_\eps(t,\xi):=\tilde{\sw}(\tfrac{t}{\eps^2},\frac{\xi}{\eps})$ satisfies 
\[
\partial_t \tilde{\sw}_\eps\approx \frac12\nabla\cdot\left( \mathsf{D} \nabla \tilde{\sw}_\eps\right), \quad\mbox{with}\quad \mathsf{D}:=\int_{\R^d}\frac{\hat{R}(p)p\otimes p}{(2\pi)^d}  dp.
\]
In other words, the second moment actually converges to the solution
of a heat equation in the high frequency regime, which is very natural
as the Schr\"odinger dynamics mixes low and high frequencies and we are looking at
an ``infinite'' long time scale. In this case, to study the behavior
of $\hat{\phi}(\tfrac{t}{\eps^2},\tfrac{\xi}{\eps})$ is a challenging
problem. To go from the jump process to a diffusion process, the
number of jumps needs to go to infinity and the effect of each individual jump
goes to zero, in other words, the effective contribution to the second
moment in the diffusive regime comes from infinitely many negligible
jumps. From a mathematical point of view, if we write
$\hat{\phi}(\tfrac{t}{\eps^2},\tfrac{\xi}{\eps})$ by a Wiener chaos
expansion, then the main contribution to its second moment comes from
those chaos of very high order, each of which goes to zero while the
sum converges as a Riemann sum. As a result, the dependence of the
wave function on the random media becomes increasingly nonlinear as
$\eps\to0$. To study the weak convergence of such random variables is
difficult, and this also appears in the study of long time behaviors of
random heat equations which requires new ideas and tools. In the weak
forcing regime we consider here, the randomness does not escape to the
tail of the chaos expansion, so it suffices to pass to the limit on the
term by term basis.
\end{remark}

\begin{remark}\label{r.multi}
While Theorem~\ref{t.mainth} is on the compensated wave function restricted to a small neighborhood of a fixed frequency $\xi$, the same proof in the paper applies to finitely many  different frequencies $\xi_1,\ldots,\xi_n$. In particular, the following result holds: if $\xi_i\neq \xi_j$ for $i\neq j\in \{1,\ldots,n\}$, then 
\[
(X_{\xi_1}^\eps(\cdot,\cdot),\ldots,X_{\xi_n}^\eps(\cdot,\cdot))\Rightarrow (X_{\xi_1}^{(1)}(\cdot, \cdot),\ldots,X_{\xi_n}^{(n)} (\cdot, \cdot))
\]
in distribution in $\underbrace{C(\bar{\bbR}_+\times\bbR^d)\times
  \cdot\times C(\bar{\bbR}_+\times\bbR^d)}_{n-\mbox{\tiny copies}}$, where
$\{X_{\xi_j}^{(j)}(\cdot,\cdot)\}_{j=1}^n$ are independent and each
component   $X_{\xi_j}^{(j)}(\cdot, \cdot)$ has the same law as that
of 
$X_{\xi_j}(\cdot, \cdot)$.
 \end{remark}

\begin{remark}
Let us discuss two related works here. In \cite{bkr}, the authors
studied the same problem except that the random potential has a smooth
temporal covariance function, which creates extra technical
difficulties in diagram expansions that we do not encounter here. The
result in \cite{bkr} is on the convergence of marginal distributions
and the proof is based on the moment convergence. Thus, the main
contribution of this paper is to show the convergence on the process level and
to present a simpler proof using the martingale structure. In \cite{gs1}, the authors considered the same weak forcing regime of the It\^o-Schr\"odinger model as ours (referred as the scintillation regime in the paper), i.e. 
\[
\text{propagation distance} =  \frac{1}{\text{size of forcing}^2},
\]
but with a low frequency initial condition, see \cite[Equations
(45)-(47)]{gs1}. One of the results in \cite{gs1} is that the fourth
moments of the wave field approximately satisfy the Gaussian summation
rule. Our conclusion is that the approximation holds also for the
respective laws.
\end{remark}

\section{Proof of Proposition \ref{prop010605-20}}
\label{sec2}

The Duhamel solution of \eqref{020405-20} is given by the following
series expansion
\begin{align}\label{e.cov3:a}
&\mathsf{w}(t,dx,\xi)=e^{-R(0)t}\Big\{|\hat{\phi}_0(\xi)|^2\delta_0(dx-\xi
 t) +
\sum_{n= 1}^{+\infty}\int_{\Delta_n(t)}d\tau_{1,n}\int_{\R^{nd}} 
|\hat{\phi}_0(p_n)|^2 \\
&
\times\delta_0\Big(dx-\xi
 \tau_1-p_1\tau_2-\ldots-p_{n-1}\tau_n-p_n(t-\tau_n)\Big)\prod_{j=1}^n
\frac{\hat{R}(p_{j-1}-p_j)}{(2\pi)^d} d\mathbf{p}_{1,n}\Big\}.\notag
\end{align}
Here $p_0=\xi$ and $d\mathbf{p}_{1,n}:=dp_1\ldots dp_n$, $d\tau_{1,n}=d\tau_1\ldots
d\tau_n$ and
$$
\Delta_n(t):=\bigg\{(\tau_1,\ldots,\tau_n):\,\tau_j\ge0,\,j=1,\ldots,n,\,\sum_{j=1}^n\tau_j\le
t\bigg\}.
$$
Using the formula
$$
\int_{\bbR^d}\delta_0\Big(dx-(\xi-p)\tau-z\Big)\frac{\hat
  R(p)dp}{(2\pi)^d}=\left(\frac{1}{2\pi \tau}\right)^d\hat
R\Big(\xi-\frac{x-z}{\tau}\Big)dx,\quad (\tau,z,\xi)\in\bbR_+\times\bbR^{2d},
$$
we conclude that, cf \eqref{010405-20a},
\begin{align}
&{\sf{u}}(t,dx,\xi)=
e^{-R(0)t}\Big\{ \left(\frac{1}{2\pi t }\right)^d|\hat{\phi}_0(\xi)|^2\hat
R\Big(\xi-\frac{x}{t}\Big)+
\sum_{n= 1}^{+\infty}\int_{\Delta_n(t)}d\tau_{1,n}\int_{\R^{nd}} 
\left(\frac{1}{2\pi \tau_1}\right)^d |\hat{\phi}_0(p_n)|^2\\
&
\times \hat
R\Big(\xi-\frac{1}{\tau_1}\big(x-p_1\tau_2-\ldots-p_{n-1}\tau_n-p_n(t-\tau_n)\big)\Big)\prod_{j=1}^n
\frac{\hat{R}(p_{j-1}-p_j)}{(2\pi)^d} d\mathbf{p}_{1,n}\Big\}dx,\notag
\end{align}
and the conclusion of the proposition follows.

\section{Proof of Theorem \ref{t.mainth}}
\label{sec3}

\subsection{Preliminaries}

The rescaled wave function $\phi_\eps(t,x)=\phi(\tfrac{t}{\eps^2},x)$
satisfies \eqref{010705-20}.
% \[
% i\partial_t\phi_\eps+\frac{1}{2\eps^2}\Delta\phi_\eps-\frac{1}{\eps}V(\frac{t}{\eps^2},x)\phi_\eps+\frac{i}{2}R(0)\phi_\eps=0, \quad\quad \phi_\eps(0,x)=\phi_0(x).
% \]
Due to the scale invariance of the white noise, see Remark
\ref{rmk1.5} above, the solution coincides, up to the law, with the
solution of the It\^o equation %. With some abuse of notation, we study the following equation instead:
\[
id\phi_\eps(t,x)+\left(\frac{1}{2\eps^2}\Delta\phi_\eps(t,x)+\frac{i}{2}R(0)\phi_\eps\right)dt-B(dt,x)\phi_\eps(t,x)=0,% \quad\quad \phi_\eps(0,x)=\phi_0(x).
\]
where $B(t,x)$ is the Wiener process with the covariance given by \eqref{B-cor}.
%the only $\eps-$dependence appears as the large factor in front of
%the Laplacian. 
We rewrite the above equation in the Fourier domain as
\begin{equation}\label{e.mainfou}
id\hat{\phi}_\eps+\left(-\frac{|\xi|^2}{2\eps^2}\hat{\phi}+\frac{i}{2}R(0)\hat{\phi}_\eps\right)dt- \int_{\R^d} \frac{\hat{B}(dt,dp)}{(2\pi)^{d}} \hat{\phi}_\eps(t,\xi-p)=0,
\end{equation}
where $\hat{B}(dt,dp)$ is the  Gaussian noise with the correlation
\begin{equation}
\label{corel}
\begin{aligned}
&\bbE\left[\hat{B}(dt,dp)
 \hat{B}^*(ds,dq)\right]=(2\pi)^d\hat R(p)\delta(t-s)\delta(p-q)dtdsdpdq,\\
&
\hat{B}^*(dt,dp)=\hat{B}(dt,-dp).
\end{aligned}
\end{equation}
Define
\begin{equation}
\label{comp}
\hat\psi_\eps(t,\xi):=\hat{\phi}_\eps(t,\xi)\exp\left\{\left(\frac{i|\xi|^2}{\eps^2}+R(0)\right)\frac{t}{2}\right\}.
%=\hat{\phi}(\frac{t}{\eps^2},\xi)\exp(i|\xi|^\alpha t+\tfrac12\eps^2R(0)t).
\end{equation}
Note that by \eqref{comp:wave}
\begin{equation}
\label{012305-20}
X^\eps_\xi(t,\eta)=\hat\psi_\eps(t,\xi+\eps^2 \eta)e^{-\frac12R(0)t}.
\end{equation}
The random field $\hat\psi_\eps(\cdot)$ is a solution of  the integral equation 
\begin{equation}\label{e.inteeq}
\hat\psi_\eps(t,\xi)=\hat{\phi}_0(\xi)+M_\eps(t,\xi).
\end{equation}
Here 
\begin{equation}
\label{M-eps}
M_\eps(t,\xi):=\frac{1}{i(2\pi)^{d}}\int_0^t\int_{\R^d}
\exp\left\{i\big(2\xi\cdot p-|p|^2\big)\frac{s}{2\eps^2}\right\}\hat\psi_\eps(s,\xi-p) \hat{B}(ds,dp).
\end{equation}
Note that
$\{M_\eps(t,\xi)\}_{t\geq0}$ is a continuous trajectory square
integrable  martingale, with respect to the filtration 
\[
\mathcal{F}_t=\sigma\left(\{B(s,x), s\in[0,t],x\in\R^d\}\vphantom{\int_0^1}\right),\quad t\ge0.
\]
Its quadratic variations are denoted by 
\begin{equation}
\label{010805-20}
\begin{aligned}
&Q^\eps_\xi(t):=\la M_\eps(\cdot,\xi),M_\eps^*(\cdot,\xi)\ra_t=\frac{1}{(2\pi)^{d}}\int_0^t\int_{\R^d} |\hat\psi_\eps(s,\xi-p)|^2 \hat{R}(p) dpds,\\
&\mathscr{Q}^\eps_\xi(t):=\la
M_\eps(\cdot,\xi),M_\eps(\cdot,\xi)\ra_t\\
&
=-\frac{1}{(2\pi)^{d}}\int_0^t \int_{\R^d} \exp\left\{-\frac{is|p|^2}{\eps^2}\right\}\hat\psi_\eps(s,\xi-p) \hat\psi_\eps(s,\xi+p) \hat{R}(p) dpds.
\end{aligned}
\end{equation}

The following simple lemma holds.

\begin{lemma}\label{l.261}
We have
\begin{equation}
\label{030705-20}
\E[|\hat\psi_\eps(t,\xi)|^2]=\tilde{\sw}(t,\xi)e^{R(0)t},\quad (t,\xi)\in\bar\bbR_+\times\bbR^d.
\end{equation}
%and
%$\E[Q_\xi^\eps(t)]=$
\end{lemma}

\begin{proof}
By the It\^o isometry, we have 
\[
\E[|\hat\psi_\eps(t,\xi)|^2]=|\hat{\phi}_0(\xi)|^2+\int_0^t\int_{\R^d}\frac{\hat{R}(p)}{(2\pi)^d}
\E[|\hat\psi_\eps(s,\xi-p)|^2] dpds.
\]
It is straightforward to check that the deterministic function
$\E[|\hat\psi_\eps(t,\xi)|^2]e^{-R(0)t}$ solves \eqref{e.eqmomentum}
- the equation satisfied by $\tilde{\sw}(t,\xi)$.
 % - which we rewrite as 
% \[
% \partial_t \tilde{\sw}(t,\xi)=\int_{\R^d}\frac{\hat{R}(p)}{(2\pi)^d} \tilde{\wW}(t,\xi-p)dp-R(0)\tilde{\sW}(t,\xi),\quad\quad \tilde{\sW}(0,\xi)=|\hat{\phi}_0(\xi)|^2.
% \]
 By the uniqueness of the solution, we conclude \eqref{030705-20}. %that $\E[|\hat\psi_\eps(t,\xi)|^2]=\tilde{\sw}(t,\xi)e^{R(0)t}$. 
% This also implies
%\[
%\E[Q_\xi^\eps(t)]=\int_0^t\int_{\R^d}\frac{\hat{R}(p)}{(2\pi)^d} \E[|\psi_\eps(s,\xi-p)|^2] dpds=\,
%\]
%which completes the proof.
\end{proof}

\begin{remark}\label{r.heuristics}
Our goal is to show that $X_\xi^\eps(t,\eta)$ converges in distribution to $X_\xi(t,\eta)$, which solves the integral equation \eqref{e.limiteq}. Given \eqref{012305-20} and \eqref{e.inteeq}, it reduces to the convergence of 
\[
M_\eps(t,\xi+\eps^2\eta)\Rightarrow \int_0^t e^{\frac12R(0)s}B_\xi(ds,\eta).
\]
The limiting Gaussianity of the martingale comes from the self averaging of the quadratic variation. Take $\eta=0$ for example. Our proof shows that $\hat{\psi}_\eps(s,\xi_1)$ and $\hat{\psi}_\eps(s,\xi_2)$ becomes asymptotically independent due to the high oscillations of the wave field, for any $\xi_1\neq \xi_2$. Thus, the term
\[
\int_{\R^d} |\hat{\psi}_\eps(s,\xi-p)|^2\hat{R}(p)dp,
\] 
appearing in the integral expression of $Q_\xi^\eps(t)$,  behaves like
a sum of independent random variables. As a result, its limit, as
$\eps\to0$, is deterministic and given by
\[
\int_{\R^d} \bbE[ |\hat{\psi}_\eps(s,\xi-p)|^2]\hat{R}(p)dp=e^{R(0)s}\int_{\R^d} \tilde{\sw}(s,\xi-p) \hat{R}(p)dp,
\] 
due to Lemma \ref{l.261}. We refer to this phenomenon as  {\em self-averaging}.
 This also explains that in order to see some nontrivial correlation structure of $\hat{\psi}_\eps$, one needs to zoom in around a fixed $\xi$, which is why we consider the process $\{\hat{\psi}_\eps(t,\xi+\eps^2 \eta)\}_{(t,\eta)\in \bar{\R}_+\times\R^d}$.
\end{remark}

%\textcolor{red}{\bf DOTAD}
In what follows we shall use the following notation:
for any set $A$ and functions $f,g:A\to\bar \bbR_+$ we say that 
$$
f(a)\les g(a),\quad a\in A,
$$
if there exists $C>0$ independent of $\eps$ such that $f(a)\le Cg(a)$, $a\in A$.

\begin{lemma}\label{l.bdq1}
We have
% quadratic variation $\left(Q_\xi^\eps(t)\right)$ (cf
% \eqref{010805-20}) satisfies
\begin{equation}
\label{020805-20}
Q_\xi^\eps(t)\le  \frac{\|R\|_{L^1(\bbR^d)} }{\|\hat R\|_{L^1(\bbR^d)} }e^{R(0)t}\|\hat{\phi}_0\|_{L^2(\R^d)}^2,\quad
(t,\xi)\in\bar\bbR_+\times\bbR^d,\,\eps\in(0,1],\quad \bbP\,\mbox{a.s.}
\end{equation}
 In consequence, for any integer $n\ge 0$ and $T>0$ %, there exists $C(T,n)>0$ such
 that 
\begin{equation}
\label{030805-20}
\sup_{t\in[0,T],\xi\in\R^d,\eps\in(0,1]}\E[|M_\eps(t,\xi)|^{2n}]<+\infty.
\end{equation}
\end{lemma}
\begin{proof}
Using the
Burkholder-Davis-Gundy inequality we conclude that
\[
\E[|M_\eps(t,\xi)|^{2n}]\leq C\E[Q_\xi^\eps(t)^n],\quad \,(t,\xi)\in\bar\bbR_+\times\bbR^d,\,\eps\in(0,1],
\] 
where the positive constant $C$ is independent of $\eps,t,\xi$.
Estimate  \eqref{030805-20} follows then directly from
\eqref{020805-20}.

To show \eqref{020805-20},
we use the fact that $R\in L^1(\R^d)$ and obtain
\begin{equation}
\label{040805-20}
\begin{aligned}
&\int_{\R^d} \frac{\hat{R}(p)}{(2\pi)^d} |\hat\psi_\eps(s,\xi-p)|^2 dp \le \frac{\|R\|_{L^1(\bbR^d)}}{(2\pi)^d}  \int_{\R^d} |\hat\psi_\eps(s,\xi-p)|^2 dp\\
&=\frac{\|R\|_{L^1(\bbR^d)} e^{R(0)s}}{(2\pi)^d}  \int_{\R^d} |\hat{\phi}_\eps(s,p)|^2 dp=\frac{\|R\|_{L^1(\bbR^d)} e^{R(0)s}}{(2\pi)^d}  \|\hat{\phi}_0\|_{L^2(\R^d)}^2,\quad (s,\xi)\in\bbR_+\times\bbR^d,
\end{aligned}
\end{equation}
where in the last step we used the conservation of the $L^2$ norm of
the solution of the Schr\"odinger equation. Estimate \eqref{020805-20}
is then a direct consequence of the first
formula of \eqref{010805-20}.
%  we conclude the deterministic bound:
% \[
% Q_\xi^\eps(t) \les \int_0^t e^{R(0)s} ds \|\hat{\phi}_0\|_{L^2(\R^d)}^2.
% \]
% It remains to apply BDG inequality to the martingale $M_\eps(\cdot,\xi)$ which gives
\end{proof}

From \eqref{e.inteeq} and \eqref{030805-20} we immediately conclude
that following.
\begin{corollary}
For any integer $n\ge0$ and $T>0$ we have %there exists $C(T,n)>0$ such
 that 
\label{c.bdpsieps}
\[
\sup_{t\in[0,T],\xi\in\R^d,\eps\in(0,1]}\E[|\hat\psi_\eps(t,\xi)|^{2n}]<+\infty.
\]
\end{corollary}

%\begin{proof}
%We have $\E[|\psi_\eps(t,\xi)|^{2n}]\les |\hat{\phi}_0(\xi)|^{2n}+\E[|M_\eps(t,\xi)|^{2n}]$. The second term is bounded by the BDG inequality:
%\[
%\E[|M_\eps(t,\xi)|^{2n}] \les \E[Q_\xi^\eps(t)^n] \les e^{nR(0)t}
%\]
%where in the last step we applied Lemma~\ref{l.bdq1}.
%\end{proof}

For $\xi\in\R^d$  fixed, we use the notation:
\[
{\cal M}_{\eps}(t,\eta):=M_\eps(t,\xi+\eps^2 \eta),\quad (t,\eta)\in\bar\bbR_+\times\R^d.
\]
The proof of Theorem~\ref{t.mainth} reduces to showing the convergence
in the law of  $\{{\cal M}_{\eps}(t,\eta)\}_{(t,\eta)\in\bar\bbR_+\times\R^d}$ over $C(\bar\bbR_+\times\R^d)$.

\subsection{Tightness}

The main result of the present section is the following estimate.
%In light of Lemma~\ref{l.bdq1}, we will use the Kolmogorov criterion to prove the tightness of $\{M_\eps\}_{\eps>0}$ as a process in $(t,\eta)$.
\begin{proposition}\label{p.tight}
For any $T>0$, $\xi\in\bbR^d$ and an integer $n\geq1$, there exists a
constant $C(T,n)>0$ such that
 \begin{equation}
\label{040805-20}
\E[|{\cal M}_{\eps} (t_1,\eta_1)-{\cal M}_{\eps} (t_2,\eta_2)|^{2n}] \leq C(T,n)(|t_1-t_2|^n+|\eta_1-\eta_2|^{2n})
\end{equation}
for all $t_1,t_2\in [0,T],$ $\eta_1,\eta_2\in\R^d$, $\eps\in(0,1]$.
\end{proposition}
From \eqref{030805-20} for any $T>0$, $\xi\in\bbR^d$ and $n\geq1$ we have
\begin{equation}
\label{012705-20}
\sup_{t\in[0,T],\eta\in\R^d,\eps\in(0,1]}\bbE[|{\cal
  M}_{\eps}(t,\eta)|^{2n}] <+\infty.
\end{equation}
As a consequence of the Kolmogorov tightness criterion for continuous
random fields, see \cite[Theorem 1.4.7, p. 38]{kunita},
\eqref{040805-20} and \eqref{012705-20} imply tightness of  the laws of 
$\{{\cal M}_{\eps} (t,\eta)\}_{(t,\eta)\in\bar\bbR_+\times\R^d}$, $\eps\in(0,1]$ over
$C(\bar\bbR_+\times\R^d)$, equipped with the standard Fr\'{e}chet topology.

%\end{proof}
To show the proposition we shall need the following.
\begin{lemma}\label{l.macon}
For any $T>0$, $\xi\in\bbR^d$ and an integer $n\geq1$, there exists a
constant $C(T,n)>0$ such that
\begin{equation}
\label{021205-20}
\E[|{\cal M}_{\eps} (t,\eta_1)-{\cal M}_{\eps} (t,\eta_2)|^{2n}]\leq C(n,T)|\eta_1-\eta_2|^{2n}
\end{equation}
for all $t\in [0,T],$ $\eta_1,\eta_2\in\R^d$, $\eps\in(0,1]$ and
\begin{equation}
\label{021205-20a}
\E[|{\cal M}_{\eps} (t_1,\eta)-{\cal M}_{\eps} (t_2,\eta)|^{2n}] \leq C(n,T)|t_2-t_1|^n
\end{equation}
for all $t_1,t_2\in [0,T],$ $\eta\in\R^d$, $\eps\in(0,1]$.
\end{lemma}

\begin{proof}
We prove \eqref{021205-20}. The argument for \eqref{021205-20a} is analogous.
The difference is written  as 
\begin{align}
\label{050805-20}
&f_\eps(t,\xi,\eta_1,\eta_2)
:=M_{\eps}  (t,\xi+\eps^2\eta_1)-M_{\eps}  (t,\xi+\eps^2\eta_2)\notag\\
&=\int_0^t\int_{\R^d}
\frac{\hat{B}(ds,dp)}{i(2\pi)^{d}}\left[\exp\left\{i\left[2(\xi+\eps^2
      \eta_1)\cdot
      p-|p|^2\right]\frac{s}{2\eps^2}\right\}-\exp\left\{i\left[2(\xi+\eps^2
      \eta_2)\cdot p-|p|^2\right]\frac{s}{2\eps^2}\right\}\right]\notag\\
&
\times\hat\psi_\eps(s,\xi+\eps^2 \eta_1-p)\\
+&\int_0^t\int_{\R^d} \frac{\hat{B}(ds,dp)}{i(2\pi)^{d}}\exp\left\{i\left[2(\xi+\eps^2
      \eta_2)\cdot p-|p|^2\right]\frac{s}{2\eps^2}\right\}[\hat\psi_\eps(s,\xi+\eps^2 \eta_1-p)-\hat\psi_\eps(s,\xi+\eps^2 \eta_2-p)]\notag\\
=&A_1+A_2, \notag
\end{align}
where $A_1$, $A_2$ denote the two integral terms appearing in \eqref{050805-20}.

%\textcolor{red}{\bf DOTAD}

\subsubsection*{Estimates of $A_1$}
Using the Burkholder-Davis-Gundy inequality and an elementary inequality $1-\cos
x\les x^2$, we can write 
\begin{equation}
\label{011205-20}
\begin{aligned}
\E[|A_1|^{2n}] \les& \E\left[  \left(\int_0^t \int_{\R^d} \hat{R}(p)[1-\cos((\eta_1-\eta_2)\cdot ps)] |\hat\psi_\eps(s,\xi+\eps \eta_1-p)|^2 dpds\right)^n\right]\\
\les & |\eta_1-\eta_2|^{2n}\E\left[
  \left(\int_0^t\int_{\R^d}|p|^2s^2\hat{R}(p)|\hat\psi_\eps(s, \xi+\eps \eta_1-p)|^2
    dpds\right)^n\right].
\end{aligned}
\end{equation}
Thanks to the fact that
$\sup_{p\in\bbR^d}|p|^2\hat{R}(p)<+\infty$ and the conservation of the $L^2$ norm of
the solution of the Schr\"odinger equation, the right hand
side of \eqref{011205-20} can be estimated by an expression of the order
$
%\big(|\eta_1-\eta_2|^{2}\red{t^2e^{R(0)t}}\|\hat\phi_0\|^{2}_{L^2(\bbR^d)}\big)^n.
|\eta_1-\eta_2|^{2n}.
$

To abbreviate  the notation we shall write
$\|X\|_{p}:=\left(\bbE|X|^p\right)^{1/p}$. Concerning $A_2$, again by the Burkholder-Davis-Gundy inequality, we have
\[
\E[|A_2|^{2n}]
\les\left\|\left(\int_0^t\int_{\R^d}\hat{R}(p)|\hat\psi_\eps(s,\xi+\eps
    \eta_1-p)-\hat\psi_\eps(s,\xi+\eps \eta_2-p)|^2dpds\right)\right\|^n_{n}.
\]
By the triangle inequality the right hand side is estimated by
\[
\left\{\int_0^t\int_{\R^d}\hat{R}(p)\|\hat\psi_\eps(s,\xi+\eps
    \eta_1-p)-\hat\psi_\eps(s,\xi+\eps
    \eta_2-p)\|_{2n}^{2}dpds\right\}^n.
\]
Invoking  \eqref{e.inteeq} we can further estimate this term by
an expression 
\[
\begin{aligned}
 &\left(\int_0^t\int_{\R^d}\hat{R}(p)[|\hat{\phi}_0(\xi+\eps\eta_1-p)-\hat{\phi}_0(\xi+\eps \eta_2-p)|^2+\|f_\eps(s,\xi-p,\eta_1,\eta_2)\|^2_{2n}]dpds\right)^n\\
\les &\eps^{2n}|\eta_1-\eta_2|^{2n}+\int_0^t\int_{\R^d} \hat{R}(p)\|f_\eps(s,\xi-p,\eta_1,\eta_2)\|_{2n}^{2n}dpds,
\end{aligned}
\] 
cf \eqref{050805-20} for the definition of $f_\eps(\cdot)$. In the
last inequality we have used the Lipschitz regularity of the initial
data and the Jensen inequality.

Combining the estimates for $A_1,A_2$, we reach the integral inequality
\[
\|f_\eps(t,\xi,\eta_1,\eta_2)\|_{2n}^{2n} \leq C|\eta_1-\eta_2|^{2n}+C\int_0^t\int_{\R^d} \hat{R}(p)\|f_\eps(s,\xi-p,\eta_1,\eta_2)\|_{2n}^{2n}dpds
\]
for all $t\in[0,T],\xi,\eta_1,\eta_2\in\R^d$, where the constant $C$
only depends on $T,n$. Taking the supremum over $\xi$ in both sides
of the inequality and invoking the Gronwall inequality we conclude the lemma.
\end{proof}

\subsubsection*{Proof of Proposition \ref{p.tight}}
Note that 
\[
|{\cal M}_{\eps}(t_1,\eta_1)-{\cal M}_{\eps} (t_2,\eta_2)|^{2n} \les |{\cal M}_{\eps} (t_1,\eta_1)-{\cal M}_{\eps}  (t_2,\eta_1)|^{2n}+|{\cal M}_{\eps}(t_2,\eta_1)-{\cal M}_{\eps}(t_2,\eta_2)|^{2n}.
\]
The conclusion of the proposition is then a straightforward consequence of
Lemma \ref{l.macon}.
\qed

\subsection{Convergence of finite dimensional distributions - limit identification}

% Recall that $X_\xi(t,\eta)$ is
%  the solution to 
% \[
% dX_\xi(t,\eta)=-\tfrac12R(0)X_\xi(t,\eta) dt+dB_\xi(t,\eta), \quad X_\xi(0,\eta)=\hat{\phi}_0(\xi),
% \]

% The solution is written explicitly as 
% \[
% X_\xi(t,\eta)=\hat{\phi}_0(\xi)e^{-\frac12R(0)t}+\int_0^t e^{-\frac12R(0)(t-s)}dB_\xi(s,\eta).
% \]
Using \eqref{comp} and \eqref{e.inteeq}, we can write the compensated
wave function, see \eqref{comp:wave} and \eqref{012305-20}, as 
\begin{equation}
\label{X-eps}
X^\eps_\xi(t,\eta)=e^{-\frac12R(0)t}\hat{\phi}_0(\xi+\eps^2
\eta)+e^{-\frac12R(0)t}{\cal M}_\eps(t,\eta).
\end{equation}
Recall that
${\cal M}_{\eps}(t,\eta)=M_\eps(t,\xi+\eps^2 \eta)$, with $\xi$ 
fixed, and $M_\eps(t,\xi)$ defined in  \eqref{M-eps}.
To complete the proof of Theorem~\ref{t.mainth}, given the tightness
proved in the previous section, we only need to show the convergence
of finite dimensional distributions of the above field. For any integer $N\ge 1$ and $\{\eta_j\}_{j=1}^N$, we will show in this section that 
\begin{equation}\label{e.conmartingale}
({\cal M}_{\eps}(t,\eta_1),\ldots, {\cal
  M}_{\eps}(t,\eta_N))\Rightarrow (Y
(t,\eta_1),\ldots,Y(t,\eta_N)),\quad \mbox{as }\eps\to0,
\end{equation}
in distribution in $C([0,\infty);\mathbb{C}^{N})$, with 
\[
Y(t,\eta):=\int_0^t e^{\frac12R(0)s} B_\xi(ds,\eta),
\]
where $B_\xi(\cdot)$ is defined in \eqref{010405-20}. The conclusion of the
theorem is then a consequence of formula \eqref{e.limiteq}.

% \begin{equation*}
% M_\eps(t,\xi)=\frac{1}{i(2\pi)^{d}}\int_0^t\int_{\R^d}
% \exp\left\{i\big(2\xi\cdot p-|p|^2\big)\frac{s}{2\eps^2}\right\}\hat\psi_\eps(s,\xi-p) \hat{B}(ds,dp).
% \end{equation*}

For $j,k=1,\ldots,N$, we define
\begin{equation}
\label{qjk}
\begin{aligned}
Q_{\eps,j,k}(t) :=\la {\cal M}_{\eps}(\cdot,\eta_j), {\cal M}_{\eps}^*(\cdot,\eta_k)\ra(t)=\int_0^t\int_{\R^d} \frac{\hat{R}(p)}{(2\pi)^d}
e^{i(\eta_j-\eta_k)\cdot ps}h_\eps^{j,k}(s,\xi-p) dpds
\end{aligned}
\end{equation}
and, 
\begin{equation}
\label{sqjk}
\begin{aligned}
\quad \mathscr{Q}_{\eps,j,k}(t) :=&\la {\cal M}_{\eps}(\cdot,\eta_j), {\cal M}_{\eps}(\cdot,\eta_k)\ra(t)\\
=&-\int_0^t\int_{\R^d} \frac{\hat{R}(p)}{(2\pi)^d} \exp\left\{i\left(\eta_j-\eta_k-\frac{p}{\eps^2}\right)\cdot p s\right\} g_\eps^{j,k}(s,\xi,p) dpds,
\end{aligned}
\end{equation}
where
\begin{align}
&
h_\eps^{j,k}(t,\xi):=\hat\psi_\eps(t,\xi+\eps^2\eta_j) \hat\psi_\eps^*(t,\xi+\eps^2
  \eta_k),\label{hjk}\\
&
g_\eps^{j,k}(t,\xi,p):=\hat\psi_\eps(t,\xi-p+\eps^2\eta_j) \hat\psi_\eps(t,\xi+p+\eps^2
  \eta_k). \label{gjk}
\end{align}

From \eqref{cyl} we also know that $\la Y
(\cdot,\eta_j),Y(\cdot,\eta_k)\ra(t)=0$ and
\begin{equation}\label{e.yjk1}
\begin{aligned}
\la Y (\cdot,\eta_j),Y^*(\cdot,\eta_k)\ra(t)=&\int_0^t e^{R(0)s}d\la B_\xi(\cdot,\eta_j),B_\xi^*(\cdot,\eta_k)\ra(s)\\
=&\int_0^t \int_{\bbR^d}e^{R(0)s}e^{-i(\eta_j-\eta_k)\cdot y}\sU(s,y+\xi
s,\xi)dy ds.
\end{aligned}
\end{equation}
Using 
\eqref{e.cov3} and \eqref{020605-20} we can further write
\begin{equation}\label{e.yjk}
\begin{aligned}
&\la Y (\cdot,\eta_j),Y^*(\cdot,\eta_k)\ra(t)\\
=&\sum_{n\geq 1}\int_{[0,t]_<^n}\int_{\R^{nd}} \prod_{\ell=1}^n \frac{\hat{R}(p_\ell)e^{i(\eta_j-\eta_k)\cdot p_\ell s_\ell}}{(2\pi)^d}|\hat{\phi}_0(\xi-p_1-\ldots-p_n)|^2d\mathbf{p}_{1,n}d\mathbf{s}_{1,n}.
\end{aligned}
\end{equation}
%By the property of $B_\xi(s,\eta)$, we know that 
%\[
%\begin{aligned}
%&Q_{j,k}(t):= \la Y_{\eta_j},Y^*_{\eta_k}\ra(t)=\\
% &\Q_{j,k}(t):=\la Y_{\eta_j},Y_{\eta_k}\ra(t)=0
%\end{aligned}
%\]
By Theorem IX.3.21 of \cite{jacod},  the proof of
\eqref{e.conmartingale} reduces to the following proposition.
\begin{proposition}\label{p.conqua}
For any $\eta_1,\ldots,\eta_N$ and $j,k=1,\ldots,N$, the processes
\begin{equation}\label{e.262}
Q_{\eps,j,k}(\cdot)\to \la Y(\cdot,\eta_j),Y^*(\cdot,\eta_k)\ra
\end{equation}
and
\begin{equation}\label{e.262a}
 \mathscr{Q}_{\eps,j,k}(\cdot)\to0,
\end{equation}
in probability in $C[0,\infty)$, as $\eps\to0$.
\end{proposition}
Establishing this result finishes the proof of Theorem \ref{t.mainth}.

\subsection{Proof of Proposition \ref{p.conqua}}
The result will be concluded at the end of a series of lemmas.
First we establish tightness property for the respective families.
\begin{lemma}\label{l.tight1}
The families of processes $\{Q_{\eps,j,k}(\cdot)\}_{\eps>0}$ and $\{\mathscr{Q}_{\eps,j,k}(\cdot)\}_{\eps>0}$ are tight in $C[0,\infty)$.
\end{lemma}

\begin{proof}
We use the Kolmogorov tightness criterion, see \cite[Theorem 12.3,
p. 95]{bil}. From the  definition of $Q_{\eps,j,k}(\cdot)$, we have
\[
|Q_{\eps,j,k}(t_2)-Q_{\eps,j,k}(t_1)|\leq \int_{t_1}^{t_2}\int_{\R^d} \frac{\hat{R}(p)}{(2\pi)^d} |h_\eps^{j,k}(s,\xi-p)|dpds
\]
for any $0\le t_1<t_2$.
Applying the triangle inequality and Corollary~\ref{c.bdpsieps}, we
conclude that for any $T>0$
\[
\begin{aligned}
\|Q_{\eps,j,k}(t_2)-Q_{\eps,j,k}(t_1)\|_n &\leq\int_{t_1}^{t_2}\int_{\R^d} \frac{\hat{R}(p)}{(2\pi)^d} \|h_\eps^{j,k}(s,\xi-p)\|_ndpds \\
&\leq C|t_2-t_1|
\end{aligned}
\]
for all $t_1,t_2\in [0,T]$, with the constant $C$ independent of $\eps$. The proof for $\{\mathscr{Q}_{\eps,j,k}(\cdot)\}_{\eps>0}$ is similar so we omit it here.
\end{proof}

With the tightness property established, we only need to show \eqref{e.262} and \eqref{e.262a} for fixed $t>0$. We will study $Q_{\eps,j,k}$ and $\mathscr{Q}_{\eps,j,k}$ separately in the following sections.

\subsubsection{Convergence of $Q_{\eps,j,k}(t)$ }

% Recall that
% \[
% Q_{\eps,j,k}(t)=\int_0^t\int_{\R^d} \frac{\hat{R}(p)}{(2\pi)^d} e^{i(\eta_j-\eta_k)\cdot ps} \psi_\eps(s,\xi+\eps^2 \eta_j-p)\psi_\eps^*(s,\xi+\eps^2\eta_k-p)dpds.
% \]
For fixed $t>0$, to show the convergence of $Q_{\eps,j,k}(t)$ in
probability, we first consider the expectation and prove the following.
\begin{lemma}\label{l.con1stmm}
For any $t>0$ and $j,k=1,\ldots,N$, we have
 \begin{equation}
\label{021305-20}
\lim_{\eps\to0}\E[Q_{\eps,j,k}(t)]= \la Y(\cdot,\eta_j),Y^*(\cdot,\eta_k)\ra(t).
\end{equation}
\end{lemma}
\begin{proof}
Fix $j,k$, define $\bar h_\eps^{j,k}(t,\xi)=\E
h_\eps^{j,k}(s,\xi)$. By \eqref{e.inteeq} we obtain the integral
equation for $\bar h_\eps^{j,k}$:
\[
\bar h_\eps^{j,k}(t,\xi)=\hat{\phi}_0(\xi+\eps^2\eta_j)\hat{\phi}_0^*(\xi+\eps^2 \eta_k)+\int_0^t \int_{\R^d} \frac{\hat{R}(p)}{(2\pi)^d} e^{i(\eta_j-\eta_k)\cdot ps}\bar h_\eps^{j,k}(t,\xi-p) dpds.
\]
Iterating the above integral equation, we obtain
\begin{equation}
\label{011305-20}
\bar h_\eps^{j,k}(t,\xi)=\hat{\phi}_0(\xi+\eps^2\eta_j)\hat{\phi}_0^*(\xi+\eps^2
\eta_k)
+\sum_{n=1}^{+\infty}\bar h_{\eps,n}^{j,k}(t,\xi),
\end{equation}
 with 
\[
\begin{aligned}
\bar h_{\eps,n}(t,\xi)&:=\int_{[0,t]_<^n}\int_{\R^{nd}} \prod_{\ell=1}^n \frac{\hat{R}(p_\ell)e^{i(\eta_j-\eta_\ell)\cdot p_\ell s_\ell}}{(2\pi)^d} \\
&\times \hat{\phi}_0(\xi+\eps^2\eta_j-p_1-\ldots-p_n)\hat{\phi}_0^*(\xi+\eps^2 \eta_k-p_1-\ldots-p_n)d \mathbf{p}_{1,n}d\mathbf{s}_{1,n}
\end{aligned}
\]
Passing to the limit in the series in the right hand side of
\eqref{011305-20} on the term by term basis and  computing
$\lim_{\eps\to0}\bar h_{\eps,n}^{j,k}(t,\xi)$, we conclude \eqref{021305-20}.
\end{proof}

\subsubsection{Convergence  of $\E |Q_{\eps,j,k}(t)|^2$}
\label{sec3.4.2}
Next we analyze $\E |Q_{\eps,j,k}(t)|^2$, which is the main technical part of the paper. The goal is to show that 
\begin{lemma}\label{l.con2ndmm}
For any $t>0$ and $j,k=1,\ldots,N$, we have 
\begin{equation}
\label{031305-20}
\E|Q_{\eps,j,k}(t)|^2\to |\la Y(\cdot,\eta_j),Y(\cdot,\eta_k)^*\ra(t)|^2,\quad
\mbox{as} \quad \eps\to0.
\end{equation}
\end{lemma} 
Combining Lemmas~\ref{l.tight1}, \ref{l.con1stmm} and \ref{l.con2ndmm}, we complete the proof of \eqref{e.262}.

The expression of $\E|Q_{\eps,j,k}(t)|^2$  involves the fourth moments of $\hat\psi_\eps$, which we will analyze through a rather standard diagram expansion. Before entering the details of the proof, we introduce the notation that will be used and prove some preliminary results.

%\[
%\begin{aligned}
%\E[|Q_{\eps,j,k}(t)|^2]=\int_{[0,t]^2}\int_{\R^{2d}}& \frac{\hat{R}(p_1)\hat{R}(p_2)}{(2\pi)^{2d}} e^{2i(\eta_j-\eta_k)\cdot (p_1s_1-p_2s_2)} \\
%&\times \E[\psi_\eps(s_1,\xi+\eps^2 \eta_j-p_1)\psi_\eps(s_2,\xi+\eps^2 \eta_k-p_2)\psi_\eps^*(s_1,\xi+\eps^2\eta_k-p_1)\psi_\eps^*(s_2,\xi+\eps^2 \eta_j-p_2)dpds
%\end{aligned}
%\]

\subsection*{Diagram expansion and moments calculation}% and the proof of Lemma~\ref{l.conpoint}}
\label{s.diagram}
Starting from the integral equation \eqref{e.inteeq}, for fixed $(t,\xi)$, we can write the random variable $\hat\psi_\eps(t,\xi)$ as an infinite Wiener chaos expansion: 
\begin{equation}%\label{e.defgn}
%\begin{aligned}
\hat\psi_\eps(t,\xi)=\sum_{n= 0}^{+\infty}  \hat\psi_{n,\eps}(t,\xi), 
\end{equation}
where  $\hat\psi_{0,\eps}(t,\xi)=\hat{\phi}_0(\xi)$ and, cf \eqref{M-eps},
\begin{equation}\label{e.defgn}
\hat\psi_{n,\eps}(t,\xi):= \int_{[0,t]_<^n}\int_{\R^{nd}}\prod_{j=1}^n \frac{\hat{B}(ds_j,dp_j)}{i(2\pi)^{d}}e^{i\Theta_{n}(\xi,\mathbf{p},\mathbf{s})\eps^{-2}}\hat{\phi}_0(\xi-p_1-\ldots-p_n) , \quad n\geq1.
%=:&\sum_{n\geq 0} g_{n,\eps}(t,\xi)
%\end{aligned}
\end{equation}

For each $n\geq1$, the phase factor is 
\begin{equation}\label{e.defphase}
\begin{aligned}
2\Theta_{n}(\xi,\mathbf{p},\mathbf{s})=&(|\xi|^2-|\xi-p_1|^2)s_1+(|\xi-p_1|^2-|\xi-p_1-p_2|^2)s_2\\
&+\ldots+(|\xi-\ldots-p_{n-1}|^2-|\xi-\ldots-p_n|^2)s_n.
\end{aligned}
\end{equation}
%with $\mathbf{p}=(p_1,\ldots,p_n), \mathbf{s}=(s_1,\ldots,s_n)$. 

%\textbf{Moments calculation and notations.}
%Throughout the rest of the section, 
In what follows we will need to estimate moments of the form 
\begin{equation}
\label{011805-20a}
\cN_\eps(\bfn,\bfn',\mathbf{t},\mathbf{t}',\pmb{\xi},\pmb{\xi}'):=\E\left[\prod_{j=1}^{N_1}\hat\psi_{n_j,\eps}(t_j,\xi_j) \prod_{j'=1}^{N_2} \hat\psi_{n'_{j'},\eps}^*(t'_{j'},\xi'_{j'})\right].
\end{equation}
Here $N_1,N_2$ are positive integers and since for our purpose it is
enough to consider the $4-$th order moments we have $N_1+N_2=4$. We
use the boldface notation, e.g. $\bfn,\pmb{\xi}$, to denote the
vectors formed by the respective elements, e.g. $\{n_j\},\{\xi_j\}$. Let $|\mathbf{n}|_1=\sum_{j=1}^{N_1} n_j$ and $|\mathbf{n}'|_1=\sum_{j'=1}^{N_2} n'_{j'}$ be the $\ell_1$ norm of $\mathbf{n},\mathbf{n}'$, and 
$K=|\mathbf{n}|_1+|\mathbf{n}'|_1.$ From the property of multiple moments of Gaussians, in order for the expression in \eqref{011805-20a} to
be  non-zero, the integer $K>0$ has to be even.  

The expression for $\hat\psi_{n,\eps}(t,\xi)$ in \eqref{e.defgn} involves
an $n-$fold stochastic time integral in the $s-$variable
and an $n-$fold integral in the momentum variable $p$. For
$\hat\psi_{n_j,\eps}(t_j,\xi_j)$, we will use
$\mathbf{s}_j=(s_{j,1},\ldots,s_{j,n_j})$ as the ``$s-$variable''
ensemble corresponding to the index $j$ and, similarly,
$\mathbf{p}_j=(p_{j,1},\ldots,p_{j,n_j})$ as the
``$p-$variable''. Similarly, we will use an analogous notation for the
primed variables.
% $\mathbf{s}'_{j'}=(s'_{j',1},\ldots,s'_{j',n'_{j'}})$ and
% $\mathbf{p}'_{j'}=(p_{j',1},\ldots,p_{j', n'_{j'}})$ for $\hat
% \psi_{n'_{j'},\eps}^*(s_\ell,\zeta_\ell)$. 
% For any ``$s-$variable'', it is associated with a ``$p-$variable'', in the sense that the two variables come from the same $\hat{V}$ in \eqref{e.defgn}.

With the above convention, we can write 
\begin{equation}\label{e.gnj}
 \begin{aligned}
 &\prod_{j=1}^{N_1} \hat\psi_{n_j,\eps}(t_j,\xi_j) =\prod_{j=1}^{N_1}\Big\{\int_{[0,t_j]_<^{n_j}}\int_{\R^{n_jd}}\prod_{k=1}^{n_j} 
\frac{\hat{B}(ds_{j,k},dp_{j,k})}{i(2\pi)^{d}}\\
&
\times \exp\left\{i\Theta_{n_j}(\xi_j,\mathbf{p}_j,\mathbf{s}_j)\eps^{-2}\right\}\hat{\phi}_0(\xi_j-p_{j,1}-\ldots-p_{j,n_j})\Big\}
 \end{aligned}
 \end{equation}
and
\begin{equation}\label{e.gmell}
\begin{aligned}
&\prod_{j'=1}^{N_2} \hat\psi_{n'_{j'},\eps}^\star(t'_{j'},\xi'_{j'})=\prod_{j'=1}^{N_2}\Big\{\int_{[0,t'_{j'}]_<^{n'_{j'}}}\int_{\R^{n'_{j'} d}}\prod_{k'=1}^{n'_{j'}} \frac{\hat{B}^*(ds'_{j',k'},dp'_{j',k'})}{-i(2\pi)^{d}}\\
&\exp\left\{-i\Theta_{n'_{j'}}(\xi'_{j'},\mathbf{p}'_{j'},\mathbf{s}'_{j'})\eps^{-2}\right\}\hat{\phi}_0^*(\xi'_{j'}-p'_{j',1}-\ldots-p'_{j',n'_{j'}}) \Big\}.
\end{aligned}
\end{equation}
%\[
%g_{n_j,\eps}(t_j,\xi_j)=\int_{[0,t_j]_<^{n_j}}\int_{\R^{n_jd}}\prod_{k=1}^{n_j} \frac{\hat{V}(t_{j,k},dp_{j,k})}{i(2\pi)^d}e^{i\Theta_{n_j}(\xi_j,p_j,t_j)\eps^{-2}}\hat{\phi}_0(\xi_j-p_{j,1}-\ldots-p_{j,n_j}) ds
%\]
%and
%\[
%g_{m_\ell,\eps}^*(s_\ell,\eta_\ell)
%\]
%With these notations, we can write 
%\[
%\begin{aligned}
%\cM_\eps(\bfn,\bfm,\mathbf{t},\mathbf{s},\pmb{\xi},\pmb{\eta})=\int_{S_{\bfn,\bfm,\mathbf{t},\mathbf{s}}}\int_{\R^{(|\bfn|_1+|\bfm|_1)d}}\prod \frac{\hat{V}(u_{j,j'},dp_{j,j'})}{i(2\pi)^d}\frac{\hat{V}^*(v_{\ell,\ell'},dq_{\ell,\ell'})}{-i(2\pi)^d}\\
%e^{i\Theta_{n_j}(\xi_j,\mathbf{p}_j,\mathbf{u}_j)\eps^{-\alpha}}e^{-i\Theta_{m_\ell}(\eta_\ell,\mathbf{q}_\ell,\mathbf{v}_\ell)\eps^{-\alpha}}
%\end{aligned}
%\]
\textbf{Pairing.} To compute
$\cN_\eps(\bfn,\bfn',\mathbf{t},\mathbf{t}',\pmb{\xi},\pmb{\xi}')$, we
need to evaluate the expectation of the $K$-th moment of the Gaussian element
\begin{equation}
\label{011805-20}
\E\left[\prod_{j=1}^{N_1} \left(\prod_{k=1}^{n_j}\hat{B}(ds_{j,k},dp_{j,k})\right) \prod_{j'=1}^{N_2} \left(\prod_{k'=1}^{n'_{j'}}\hat{B}^*(ds'_{j',k'},dp'_{j',k'})\right)\right]
\end{equation}
that we handle using the Wick theorem, see e.g. \cite[Theorem
1.36]{janson}. To apply it we introduce some further notations.

Suppose that $A$ is a finite  subset of even
  cardinality. By a pairing ${\cal P}$ over the elements of the set, we mean
  any partition of $A$ into two element disjoint subsets.
Consider the set of all pairs $(\lam,w)$ belonging to the set 
\[
{\cal Z}:=\{(s_{j,k},p_{j,k}),(s'_{j',k'},
p'_{j',k'})\}_{j,k, j',k'}
\] ordered by the lexicographical
order, i.e. $(\lam_1,w_1)$ precedes $(\lam_2,w_2)$ and we write  $(\lambda_1,w_1)\preceq (\lambda_2,w_2)$,   if any
of the following happens:
\begin{itemize}
\item[1)] $(\lam_1,w_1)= (s_{j,k},p_{j,k})$ and  $(\lam_2,w_2)=
  (s'_{j',k'},p'_{j',k'})$, 
\item[2)]
%$(\lam_1,w_1)= (s_{j,k},p_{j,k})$  and  $(\lam_2,w_2)=
%(s_{j',k'},p_{j',k'})$ but $j<j'$, or 
$(\lambda_\ell,w_\ell)=(s_{j_\ell,k_\ell},p_{j_\ell,k_\ell})$, $\ell=1,2$ and $j_1<j_2$,
\item[3)] $(\lam_\ell,w_\ell)=
(s'_{j'_\ell,k'_\ell},p'_{j'_\ell,k'_\ell})$, $\ell=1,2$ and
$j'_1<j'_2$.
\end{itemize}
Consider all ordered pairings formed over the set
$\{(s_{j,k},p_{j,k}),(s'_{j',k'}, p'_{j',k'})\}_{j,k, j',k'}$ such that two elements with the same $j$-s, or
$j'$-s cannot be paired. In other words,
there is no pair formed inside any vector 
$$
[(s_{j,1},p_{j,1}),\ldots, (s_{j,n_j},p_{j,n_j})],\quad\mbox{ or }\quad
 [(s_{j',1}',p_{j',1}'),\ldots, (s_{j',n'_{j'}}',p_{j',n'_{j'}}')].
$$ 
Denote $((\lambda,w),(\lambda',w'))$  a typical pair. 
Assume that $(\lambda,w)$ is the left element of the pair 
in the sense that $(\lambda,w)\preceq (\lambda',w')$.
Each pairing also can be easily ordered
with the order inherited from the order of the
set ${\cal Z}$ on its left vertices.

Denote by $\Pi$ the set of all ordered pairings as described
above. 
Let 
\[
\iota_{w,w'}=\left\{
\begin{array}{ll}
1,&\mbox{ if $(\lam,w)= (s_{j,k},p_{j,k})$  and  $(\lam',w')=
    (s_{j',k'}',p_{j',k'}')$},\\
&\\
0,& \mbox{ if otherwise,}%&\mbox{ if $(\lam_\ell,w_\ell)=
    % (s_{j_\ell,k_\ell},p_{j_\ell,k_\ell})$, $\ell=1,2$, or $(\lam_\ell,w_\ell)=
%(s'_{j'_\ell,k'_\ell},p'_{j'_\ell,k'_\ell})$, $\ell=1,2$,}
\end{array}
\right.
\]
and
\[
S_{\bfn,\bfn',\mathbf{t},\mathbf{t}'}:=[0,t_1]_<^{n_1}\times \ldots \times  [0,t_{N_1}]^{n_{N_1}}_<\times [0,t_1']_<^{n_1'}\times \ldots \times  [0,t_{N_2}']_<^{n_{N_2}'}.
\]
Using the Wick theorem we can write that
\begin{align}
\label{021805-20}
&
\cN_\eps(\bfn,\bfn',\mathbf{t},\mathbf{t}',\pmb{\xi},\pmb{\xi}') =\sum_{{\cal P}\in \Pi}\int_{\R^{Kd}}\int_{S_{\bfn,\bfn',\mathbf{t},\mathbf{t}'}}\\
&
\left\{\prod_{((\lambda,w),(\lambda',w'))\in {\cal P}}\frac{\hat{R}(w)}{(2\pi)^{d}}\delta(\lambda-\lambda')\delta\big(w+(-1)^{\iota_{w,w'}}w'\big)\right\}e^{i\eps^{-2}\Theta} \Phi \, d \pmb{\lam}d \pmb{\lam}'d\mathbf{w}d\mathbf{w}',\notag
\end{align}
where to ease the notation, we write
$$
d\pmb{\lam}d\pmb{\lambda}'=\prod_{j=1}^{N_1}\prod_{k=1}^{n_j}ds_{j,k}\prod_{j'=1}^{N_2}\prod_{k'=1}^{n_{j'}'}ds'_{j',k'}\quad\mbox{and}\quad 
d\mathbf{w}d\mathbf{w}'=\prod_{j=1}^{N_1}\prod_{k=1}^{n_j}dp_{j,k}\prod_{j'=1}^{N_2}\prod_{k'=1}^{n_{j'}'}dp'_{j',k'}.
$$
Performing the integration over the $(\lam',w')$-variables, we conclude that%\gucomment{the product $\prod_{\ell=1}^{K/2}\hat{R}(w_\ell)$ is not very clear. Do we re-introduce those $w_1,\ldots,w_{K/2}$ variables? Or maybe we can keep the original notation to write it as $\prod_{((\lambda_1,w_1),(\lambda_2,w_2))\in \mathcal{P}}\frac{\hat{R}(w_1)}{(2\pi)^d}$}
\begin{align}
\label{021805-20a}
\cN_\eps(\bfn,\bfn',\mathbf{t},\mathbf{t}',\pmb{\xi},\pmb{\xi}')
=\sum_{{\cal P}\in \Pi}\int_{\R^{Kd/2}}\int_{S_{\cal P}}
\left\{\prod_{\ell=1}^{K/2}\frac{\hat{R}(w_\ell)}{(2\pi)^{d}}\right\}e^{i\eps^{-2}\Theta_{\cal P}}
  \Phi_{\cal P} \, d \pmb{\lam}d\mathbf{w},
\end{align}
%\red{
%\begin{equation}
%\begin{aligned}
%\label{021805-20a}
%\cN_\eps(\bfn,\bfn',\mathbf{t},\mathbf{t}',\pmb{\xi},\pmb{\xi}')
%=\sum_{{\cal P}\in \Pi}\int_{\R^{Kd/2}}\int_{S_{\cal P}}
%\left\{\prod_{((\lambda_1,w_1),(\lambda_2,w_2))\in {\cal P}}\frac{\hat{R}(w_1)}{(2\pi)^{d}}\right\} e^{i\eps^{-2}\Theta_{\cal P}}
%  \Phi_{\cal P} \, d \pmb{\lam}d\mathbf{w},
%\end{aligned}
%\end{equation}
%}
where 
the domain of integration $S_{\mathcal{P}}$ - that is a convex set, the phase $\Theta_{\cal
  P}$ and the factor $\Phi_{\cal P}$ in \eqref{021805-20a} are induced from
$S_{\bfn,\bfm,\mathbf{t},\mathbf{s}}$, 
$\Theta$ and $\Phi$ by the collapse of  $\lambda'$ and $w'$ -
variables. We have also denoted all $\pmb{\lambda}-$variables by $\{\lambda_\ell\}$ and all $\mathbf{w}-$variables by $\{w_\ell\}$, and assume that $\{\lambda_1,\ldots,\lambda_{K/2}\}$ is ordered according to the pairing order. Note that the components of $\pmb{\lambda}$ depend on the
partition $\mathcal{P}$. However they  obey the
ordering inherited from  each individual simplex.
 % In addition,  $S_{\mathcal{P}}$ clearly depends on the partition
 % $\mathcal{P}$, as each $\lambda_1$ is associated with a
 % $\lambda_2-$variable which also obeys the ordering before being
 % integrated out. 
It is quite possible that $S_{\mathcal{P}}=\emptyset$ for some pairings, e.g. if we have the domain $[0,t_1]_<^2\times [0,t_1']_<^2$, and the pairing
\[
\{(s_{1,1},p_{1,1}),(s'_{1,2},p'_{1,2})),((s_{1,2},p_{1,2}),(s'_{1,1},p'_{1,1}))\},
\]
then the set $\{s_{1,1}>s_{1,2}, s'_{1,1}>s'_{1,2}\}$ does not
intersect with $\{s_{1,1}=s'_{1,2},s_{1,2}=s'_{1,1}\}$ and, as a
result, $S_{\mathcal{P}}=\emptyset$. 

 For each $\mathcal{P}$, we can partition
$S_{\mathcal{P}}$, up to a null  Lebesgue measure  set, into subsets
depending on the ordering of $\{\lambda_1,\ldots,\lambda_{K/2}\}$. More precisely,
given a permutation $\si$ of the set $\{1,\ldots,\tfrac{K}{2}\}$, we
define 
\[
S_{\mathcal{P},\sigma}:=S_{\mathcal{P}}\cap \{\lambda_{\sigma(1)}> \lambda_{\sigma(2)}> \ldots> \lambda_{\sigma(\frac{K}{2})}\}.
\]
Obviously the sets are disjoint for different permutations and 
$
{\sf m}\left(S_{\mathcal{P}}\setminus\cup_{\sigma} S_{\mathcal{P},\sigma}\right)=0$, 
where the union $\cup_\sigma$ is taken over all the permutations and $\mathsf{m}(\cdot)$
denotes the Lebesgue measure on $\bbR^{K/2}$.

As in the case of ${\cal S}_{\cal P}$, it  is entirely possible
that some $S_{\mathcal{P},\sigma}$ can be empty sets -- we will take
care of them later in the proof, and for the moment, we do not
distinguish between them to keep the notation simple.
With the above notation, we write:
\begin{equation}
\label{021805-20b}
\cN_\eps(\bfn,\bfn',\mathbf{t},\mathbf{t}',\pmb{\xi}, \pmb{\xi}') = \sum_{\mathcal{P}}\sum_{\sigma}\mathcal{I}_\eps(\mathcal{P},\sigma)
\end{equation} with %\gucomment{the same problem with $\prod_{\ell=1}^{K/2}\hat{R}(w_\ell)$}
\begin{align}
\label{IPS}
\mathcal{I}_\eps(\mathcal{P},\sigma):=\int_{\R^{Kd/2}}\int_{S_{\cal P,\si}}\left\{\prod_{\ell=1}^{K/2}\frac{\hat{R}(w_\ell)}{(2\pi)^{d}}\right\}e^{i\eps^{-2}\Theta_{\cal P}}
  \Phi_{\cal P} \, d \pmb{\lam}d\mathbf{w} .
\end{align}
%\red{
%\begin{align}
%\label{IPS}
%\mathcal{I}_\eps(\mathcal{P},\sigma):=\int_{\R^{Kd/2}}\int_{S_{\cal P,\si}}\left\{\prod_{((\lambda_1,w_1),(\lambda_2,w_2))\in {\cal P}}\frac{\hat{R}(w_1)}{(2\pi)^{d}}\right\}e^{i\eps^{-2}\Theta_{\cal P}}
%  \Phi_{\cal P} \, d \pmb{\lam}d\mathbf{w} .
%\end{align}
%}

%\textcolor{red}{\bf DOTAD}

\subsection*{Bounds on $\cN_\eps(\bfn,\bfn',\mathbf{t},\mathbf{t}',\pmb{\xi}, \pmb{\xi}')$}

Let $\{W_t\}_{t\geq 0}$ be a standard real-valued Brownian
motion. Define the processes
\begin{equation}\label{e.tildeg}
\begin{aligned}
&H_0(t)\equiv 1,\\
& H_n(t)=R(0)^{\frac{n}{2}}\int_{[0,t]_<^n}dW_{s_n}\ldots dW_{s_1}=\frac{1}{n!}\left({R(0)}{t}\right)^{n/2}h_n\left(\frac{W_t}{\sqrt{t}}\right), \quad n\geq1,
\end{aligned}
\end{equation}
where $h_n(x):=(-1)^ne^{x^2/2}(e^{-x^2/2})^{(n)}$ is the $n$-th degree
Hermite polynomial, cf (3.3.8) of \cite{oksendal}.
%with the convention $\tilde{g}_0\equiv1$. 
\begin{lemma}\label{l.sumbd}
The following estimate holds 
\begin{equation}\label{e.mmbd1}
\left|\cN_\eps(\bfn,\bfn',\mathbf{t},\mathbf{t}',\pmb{\xi}, \pmb{\xi}')\right| \leq \|\hat{\phi}_0\|_{L^\infty(\R^d)}^{4}\E\left[\prod_{j=1}^{N_1} H_{n_j}(t_j) \prod_{j'=1}^{N_2}H_{n_{j'}'}(t'_{j'}) \right]
\end{equation}
for all $(\bfn,\bfn',\mathbf{t},\mathbf{t}',\pmb{\xi}, \pmb{\xi}')$.
In  consequence,
\begin{equation}\label{e.mmsum}
\begin{aligned}
&\sum_{\bfn,\bfn'}\left|\E\left[\prod_{j=1}^{N_1}\hat\psi_{n_j,\eps}(t_j,\xi_j) \prod_{j'=1}^{N_2}\hat\psi^\star_{n'_{j'},\eps}(t_{j'}',\xi'_{j'})\right] \right|\leq \|\hat{\phi}_0\|_{L^\infty(\R^d)}^{4}\exp\left\{6R(0)\max_{j,j'}(t_j, t_{j'}')\right\}
\end{aligned}
\end{equation}
for all $(\mathbf{t},\mathbf{t}',\pmb{\xi}, \pmb{\xi}')$.
\end{lemma}
\begin{proof}
We use expression \eqref{021805-20a} to respresent
$\cN_\eps(\bfn,\bfn',\mathbf{t},\mathbf{t}',\pmb{\xi},
\pmb{\xi}')$. Recall that $N_1+N_2=4$. Thanks to the obvious bounds  $|e^{i\eps^{-2}\Theta_{\mathcal{P}}}|\leq1$ and
$|\Phi_{\mathcal{P}}|\leq \|\hat{\phi}_0\|_{L^\infty(\R^d)}^{4}$,  we get
\[
\begin{aligned}
&\left|\cN_\eps(\bfn,\bfn',\mathbf{t},\mathbf{t}',\pmb{\xi},
\pmb{\xi}')\right| \\
&\leq \|\hat{\phi}_0\|_{L^\infty(\R^d)}^{4}\sum_{{\cal P}\in \Pi}\int_{\R^{Kd/2}}\int_{S_{\cal P}}
\left\{\prod_{\ell=1}^{K/2}\frac{\hat{R}(w_\ell)}{(2\pi)^d}\right\} \, d \pmb{\lam}d\mathbf{w}\\
&=\|\hat{\phi}_0\|_{L^\infty(\R^d)}^4[R(0)]^{K/2}\sum_{{\cal P}\in
  \Pi}{\sf m} (S_{\cal P}).
\end{aligned}
\]
Applying the Wick formula,
we also conclude that 
\[
[R(0)]^{K/2}\sum_{{\cal P}\in
  \Pi }{\sf m}(S_{\cal P})=\E\left[\prod_{j=1}^{N_1} H_{n_j}(t_j) \prod_{j'=1}^{N_2}H_{n_{j'}'}(t'_{j'}) \right].
\]
 As a result \eqref{e.mmbd1} follows.
To prove \eqref{e.mmsum} we use the well-known formula, see
e.g. \cite[formula (1.1), p. 4]{nualart}
\[
\sum_{n=0}^{+\infty}H_n(t)=e^{\sqrt{R(0)}W_t-\frac12R(0)t}=:\mathcal{E}_t,
\]
where the convergence holds both a.s. and in the $L^p$-sense for any $p\in[1,+\infty)$. Thus, 
\[
\begin{aligned}
&\sum_{\bfn,\bfn'}\E\left[\prod_{j=1}^{N_1} H_{n_j}(t_j) \prod_{j'=1}^{N_2}H_{n_{j'}'}(t'_{j'}) \right]=\E\big[\prod_{j=1}^{N_1}\mathcal{E}_{t_j}\prod_{j'=1}^{N_2}\mathcal{E}_{t'_{j'}}\big]\\
&\leq \exp\left\{\tfrac{(N_1+N_2)(N_1+N_2-1)}{2}R(0)\max_{j,j'}(t_j,t'_{j'})\right\}=\exp\left\{6R(0)\max_{j,j'}(t_j,t'_{j'})\right\},%=e^{(2N^2-N)R(0)t},
\end{aligned}\]
which completes the proof.
\end{proof}

\subsubsection{Proof of Lemma~\ref{l.con2ndmm}}
Fix $t>0,j,k=1,\ldots,N$. Recall that 
\[
Q_{\eps,j,k}(t)=\int_0^t\int_{\R^d} \frac{\hat{R}(p)}{(2\pi)^d} e^{i(\eta_j-\eta_k)\cdot ps} \hat\psi_\eps(s,\xi+\eps^2 \eta_j-p) \hat\psi_\eps^*(s,\xi+\eps^2\eta_k-p)dpds.
\]
%As the variables $\xi,\eta_j,\eta_k,t$ are fixed, 
We write its second moment as 
\begin{equation}
\label{022205-20}
\begin{aligned}
\E[|Q_{\eps,j,k}(t)|^2]=\int_{[0,t]^2}\int_{\R^{2d}} \frac{\hat{R}(p)\hat{R}(p')}{(2\pi)^{2d}} e^{i(\eta_j-\eta_k)\cdot (ps-p's')}\M_{4,\eps}(s,s',p,p')dpdp'dsds',
\end{aligned}
\end{equation}
with
\[
\begin{aligned}
\M_{4,\eps}(s,s',p,p')=\E[&\hat\psi_\eps(s,\xi+\eps^2 \eta_j-p) \hat\psi_\eps^*(s,\xi+\eps^2\eta_k-p)\\
\times&\hat\psi_\eps^*(s',\xi+\eps^2 \eta_j-p') \hat\psi_\eps(s',\xi+\eps^2\eta_k-p')].
\end{aligned}
\]
By Corollary~\ref{c.bdpsieps}, we know that 
\[
|\M_{4,\eps}(s,s',p,p')| \leq C(t),
\]
therefore, to study the limit of $\E[|Q_{\eps,j,k}(t)|^2]$, by the
dominated convergence theorem, we only need to analyze the limit of
$\M_{4,\eps}(s,s',p,p')$ as $\eps\to0$, for   a.e. $s,s',p,p'$ in the
respective Lebesgue measure. Define 
\[
\M_{2,\eps}(s,p)=\E[\hat\psi_\eps(s,\xi+\eps^2 \eta_j-p) \hat\psi_\eps^*(s,\xi+\eps^2\eta_k-p)].
\]
Using \eqref{011305-20} with $\hat{\phi}_0(\xi-p)$ in place of
$\hat{\phi}_0(\xi)$, we conclude that 
\[
\begin{aligned}
\lim_{\eps\to0}\M_{2,\eps}(s,p)=|\hat{\phi}_0(\xi-p)|^2
+\sum_{n=1}^{+\infty} \int_{[0,s]_<^n} \int_{\R^{nd}} &\prod_{\ell=1}^n \frac{\hat{R}(p_\ell)e^{i(\eta_j-\eta_k)\cdot p_\ell s_\ell}}{(2\pi)^d}\\
&\times |\hat{\phi}_0(\xi-p-p_1-\ldots-p_n)|^2 d\mathbf{p}_{1,n}d\mathbf{s}_{1,n}
\end{aligned}
\]
for each $s,p$.
%Recall that $\la Y(,\cdot,{\eta_j}),Y^*(\cdot,{\eta_k})\ra_t$ is given in \eqref{e.yjk},
%\[
%\M_{4,\eps}(\mathbf{p},\mathbf{s})-\prod_{n=1}^2\M_{2,\eps}(p_j,s_j)\to0, \quad\quad \mbox{ if } p_1\neq p_2
%\]
%which is implied by the following lemma.
 The conclusion of Lemma \ref{l.con2ndmm}  then
 is a consequence of Lemma \ref{l.con1stmm} and 
the following. 
\begin{lemma}\label{l.com4n2}
For any $s,s',p,p'$ such that $p\neq p'$, we have 
\begin{equation}
\label{032005-20}
\M_{4,\eps}(s,s',p,p')-\M_{2,\eps}(s,p)\M_{2,\eps}^*(s',p')\to0, \quad \mbox{ as } \eps\to0.
\end{equation}
\end{lemma}
%\textcolor{red}{\bf DOTAD}

\begin{proof}
We let
\begin{equation}
\label{012005-20}
\begin{aligned}
&t_1=t_1'=s,\quad t_2=t_2'=s', \\
&\xi_1=\xi+\eps^2\eta_j-p, \quad \xi_2=\xi+\eps^2\eta_k-p', \quad \xi'_1=\xi+\eps^2\eta_k-p,\quad \xi'_2=\xi+\eps^2\eta_j-p',
\end{aligned}
\end{equation}
in the diagram expansion \eqref{021805-20} and \eqref{021805-20b}.
Thus,
\begin{equation}
\label{022005-20}
\begin{aligned}
\M_{4,\eps}(s,s',p,p')=&\E[\hat\psi_\eps(t_1,\xi_1) \hat\psi_\eps(t_2,\xi_2) \hat\psi_\eps^*(t_1',\xi_1')\hat\psi_\eps^*(t_2',\xi_2')]\\
=&\sum_{\bfn,\bfn'}\cN_\eps(\bfn,\bfn',\mathbf{t},\mathbf{t}',\pmb{\xi},\pmb{\xi}').
\end{aligned}
\end{equation}
Here
$\cN_\eps(\bfn,\bfn',\mathbf{t},\mathbf{t}',\pmb{\xi},\pmb{\xi}')$
is given by \eqref{011805-20a}:
\[
\cN_\eps(\bfn,\bfn',\mathbf{t},\mathbf{t}',\pmb{\xi},\pmb{\xi}')=\E\left[\prod_{j=1}^{N_1}\hat\psi_{n_j,\eps}(t_j,\xi_j) \prod_{j'=1}^{N_2} \hat\psi_{n'_{j'},\eps}^*(t'_{j'},\xi'_{j'})\right].
\]
with $N_1=N_2=2$ and
$\mathbf{t},\mathbf{t}',\pmb{\xi},\pmb{\xi}'$ determined from
\eqref{012005-20}.

By Lemma~\ref{l.sumbd}, while computing the limit of
$\M_{4,\eps}(s,s',p,p')$, as $\eps\to0$, we can enter with the limit
under the series in the right hand side of \eqref{022005-20}. So
the question reduces to the computation of the limits of
each 
$\cN_\eps(\bfn,\bfn',\mathbf{t},\mathbf{t}',\pmb{\xi},\pmb{\xi}')$.

%  Since $p\neq p'$,  
% we have $\lim_{\eps\to0}(\xi_j-\xi_{j'}),j=1,2$, all other pairs satisfy $\lim_{\eps\to0} |a-b|=|p-p'|\neq 0$.
 
% On the other hand, we note that  
We have
\[
\M_{2,\eps}(s,p)\M_{2,\eps}^*(s',p')=\prod_{j=1}^2\sum_{n=0}^{+\infty}\E[\hat\psi_{n,\eps} (t_j,\xi_j)\hat\psi^*_{n,\eps} (t_j',\xi_j')]],
\]
thus, to complete the proof of \eqref{032005-20}, it suffices to show that
 \begin{equation}
\label{042005-20}
\lim_{\eps\to0}\left\{
  \cN_\eps(\bfn,\bfn',\mathbf{t},\mathbf{t}',\pmb{\xi},\pmb{\xi}')-\delta_{n_1,n_1'}\delta_{n_2,n'_2}\prod_{j=1}^2\E[\hat
  \psi_{n_j,\eps}(t_j,\xi_j) \hat
  \psi_{n_j,\eps}^*(t_j',\xi_j')]\right\}=0,
 \end{equation}
where $\delta_{n,m}$ is the Kronecker symbol.

% Using the diagram expansion and \eqref{021805-20b}, we have $\cM_\eps(\bfn,\bfm,\mathbf{t},\mathbf{s},\pmb{\xi},\pmb{\zeta}) = \sum_{\sigma}\mathcal{I}_\eps(\sigma)$ with 
% \[
% \mathcal{I}_\eps(\sigma):= \int_{\R^{Kd}}\int_{\sigma_{\bfn,\bfm,\mathbf{t},\mathbf{s}}}\left(\prod_{(\lambda_1,\lambda_2)\in \mathcal{P}_\sigma}\frac{\hat{R}(w_1)}{(2\pi)^{d}}\delta(\lambda_1-\lambda_2)\delta(w_1+w_2)\right)\cdot e^{i\eps^{-2}\Theta}\cdot \Phi \, d\mathbf{u}d\mathbf{v}d\mathbf{p}d\mathbf{q}
% \]
Consider the diagram expansion \eqref{021805-20b}  for
$\cN_\eps(\bfn,\bfn',\mathbf{t},\mathbf{t}',\pmb{\xi},\pmb{\xi}')$, which we write as a \emph{finite} sum:
\[
\cN_\eps(\bfn,\bfn',\mathbf{t},\mathbf{t}',\pmb{\xi}, \pmb{\xi}') = \sum_{\mathcal{P}}\sum_{\sigma}\mathcal{I}_\eps(\mathcal{P},\sigma).
\]
Among all the ordered pairings, we distinguish one
special called the {\em
ladder pairing}, defined in the case of $n_j=n_j'$ for
$j=1,2$, as follows
\begin{align*}
{\cal P}_{\rm lad}:=&\Big\{\big((s_{1,1},p_{1,1}),
  (s'_{1,1},p'_{1,1})\big),\ldots, \big((s_{1,n_1},p_{1,n_1}),
  (s'_{1,n_1},p'_{1,n_1})\big),\\
& \big((s_{2,1},p_{2,1}), (s'_{2,1},p'_{2,1})\big),\ldots, \big((s_{2,n_2},p_{2,n_2}), (s'_{2,n_2},p'_{2,n_2})\big)\Big\}.
\end{align*}

% Define $S_{\mathrm{per}}$ as the set of all permutations $\sigma$ such
% that the corresponding pairing $\mathcal{P}_\sigma$ contains only
% $\{(u_{1,j},v_{1,j}),(u_{2,\ell}, v_{2,\ell})\}_{j,\ell}$., so we define $S_{\mathrm{per}}=\emptyset$ if this is not case.

% Consider the case of $S_{\mathrm{per}}\neq \emptyset$, it is straightforward to check that 
% %if $\sigma\in S_{\mathrm{per}}$, the phase factor $\Theta$, given by \eqref{e.Theta}, equals to zero after integrating out the Dirac functions 
% %\[
% %\prod_{(\lambda_1,\lambda_2)\in \mathcal{P}_\sigma}\delta(\lambda_1-\lambda_2)\delta(w_1+w_2),
% %\] and  
% after a summation over all permutations in $S_{\mathrm{per}}$, we
% have (recall that $t_1=s_1,t_2=s_2$)
Using the diagram expansion to represent
 $\E[\hat
  \psi_{n_j,\eps}(t_j,\xi_j) \hat
  \psi_{n_j,\eps}^*(t_j',\xi_j')]$, $j=1,2$ we conclude that, cf \eqref{021805-20b},
\[
\begin{aligned}
\sum_{\sigma}\mathcal{I}_\eps({\cal P}_{\rm lad},\sigma)=      \prod_{j=1}^2\E[\hat
  \psi_{n_j,\eps}(t_j,\xi_j) \hat
  \psi_{n_j,\eps}^*(t_j',\xi_j')].
\end{aligned}
\]
Next we show that
\begin{equation}\label{e.4223}
\lim_{\eps\to0}\mathcal{I}_\eps({\cal P},\sigma)=0,
\quad\quad   \mbox{ if } {\cal P}\neq {\cal P}_{\rm lad}.
\end{equation}
This would end the proof of \eqref{042005-20}, finishing in this way
the proof of Lemma \ref{l.com4n2}.

\subsection*{Proof of (\ref{e.4223})}
Since in our argument only the time components of  the paired elements
$(\lam,w)$ (see the definition of a pairing)
play a role, to simplify the notation, when speaking about
${\cal P}$ we shall refer to the pairing between the $\lam$  (temporal) components only.
 Let us consider a  pairing 
\begin{equation}
\label{022205-20}
{\cal P}\neq {\cal P}_{\rm lad}
\end{equation}
 and
 suppose that 
\begin{equation}\label{e.4223n}
\lim_{\eps\to0}\mathcal{I}_\eps({\cal P},\sigma)\neq 0.
\end{equation} 
All the
 paired $\lam$-s come from the set
of variables 
\[
\begin{aligned}
\{s_{1,1},\ldots,s_{1,n_1},s_{2,1},\ldots,s_{2,n_2},s'_{1,1},\ldots,s'_{1,n_1'},s_{2,1}',\ldots,s_{2,n_2'}'\},
\end{aligned}
\]
and the partition is ${\cal P}=\{(\lam_1, \lam_1'),\ldots, (\lam_{K/2},
\lam_{K/2}')\}$.
%\red{where with a  slightly abused notation, we let $\{\lambda_\ell\}_{\ell=1,\ldots,K/2}$ denote all $\lambda_1-$variables after the ``collapsing''. In the following we will use $\{w_\ell\}_{\ell=1,\ldots,K/2}$ to denote the corresponding $w_1-$variables.
%}
Recall that $\{\lambda_1,\ldots,\lambda_{K/2}\}$  are ordered according to
the pairing ordering, and $\{w_1,\ldots,w_{K/2}\}$ are the corresponding momentum variables. For any permutation $\sigma$, we defined
\[
S_{\mathcal{P},\sigma}=S_{\mathcal{P}}\cap \big\{\lambda_{\sigma(1)}> \lambda_{\sigma(2)}> \ldots> \lambda_{\sigma(\frac{K}{2})}\big\},
\]
and 
\[
\mathcal{I}_\eps(\mathcal{P},\sigma)=\int_{\R^{Kd/2}}\int_{S_{\cal P,\si}}\left\{\prod_{\ell=1}^{K/2}\frac{\hat{R}(w_\ell)}{(2\pi)^{d}}\right\}e^{i\eps^{-2}\Theta_{\cal P}}
  \Phi_{\cal P} \, d \pmb{\lam}d\mathbf{w}.
  \]
As we will integrate $\lambda_{\sigma(1)},\lambda_{\sigma(2)},\ldots$ in order, to ease the notation, we perform a change of variable so that after the change we have $\lambda_1>\ldots>\lambda_{K/2}$. More precisely, we change 
\begin{equation}\label{e.281}
\lambda_\ell,w_\ell\mapsto \lambda_{\sigma^{-1}(\ell)}, w_{\sigma^{-1}(\ell)}, \quad\quad \ell=1,\ldots,\tfrac{K}{2},
\end{equation}
and let
$$
\tilde{\cal S}_{{\cal P},\si}:=\left\{(\lam_1,\ldots,\lam_{K/2}):\,
  (\lam_{\si^{-1}(1)},\ldots,\lam_{\si^{-1}(K/2)})\in {\cal S}_{{\cal P},\si}\right\}.
$$ 
It is clear that  $\lam_1> \ldots >\lam_{K/2}$ for any $(\lam_1,\ldots,\lam_{K/2})\in \tilde{\cal S}_{{\cal P},\si}$.
%For the fixed partition $\mathcal{P}$ and permutation $\sigma$, let 
%\[
%(\lambda_1,\ldots,\lambda_{K/2})=(r_{\sigma(1)},\ldots,r_{\sigma(K/2)}),
%\]
%In other words, inside the set $\tilde{\cal S}_{\mathcal{P},\sigma}$, $\lambda_1$ to be the largest temporal variable, $\lambda_2$ to be the second largest, etc. 
Now we can write $\mathcal{I}_\eps(\mathcal{P},\sigma)$  as 
\begin{equation}\label{e.261}
\mathcal{I}_\eps(\mathcal{P},\sigma)=\int_{\R^{Kd/2}}\int_{\tilde{S}_{\cal P,\si}}\left\{\prod_{\ell=1}^{K/2}\frac{\hat{R}(w_\ell)}{(2\pi)^{d}}\right\}e^{i\eps^{-2}\Theta_{\cal P}}
  \Phi_{\cal P} \, d \pmb{\lam}d\mathbf{w},
  \end{equation}
  where we have changed variables in $\Theta_{\cal P},\Phi_{\cal P}$ according to \eqref{e.281}.
  
  Let us first present a rough sketch of the proof. In the expression \eqref{e.261}, the phase factor $\Theta_{\cal P}$ is a linear combination of $\lambda_1,\ldots,\lambda_{K/2}$. As we will see later, the fact that ${\cal P}\neq {\cal P}_{\rm lad}$ induces a nonzero order $O(1)$ coefficient associated with some $\lambda_\ell$, so we write $\Theta_{\cal P}=\theta_{\mathcal{P}}\lambda_\ell+\tilde{\Theta}_{\mathcal{P}}$ for some $\theta_{\mathcal{P}}\neq 0$. Using the elementary fact that 
  \[
\sup_{a,b\in [0,T]}  \int_a^b e^{i\eps^{-2} \theta_{\cal P} \lambda_\ell} d\lambda_\ell\to0, \quad \mbox{for any $T>0$, provided that $\theta_{\cal P}\neq0$,}
 \]
 we derive that 
 \[
 \int_{\tilde{S}_{\cal P,\si}}e^{i\eps^{-2}\Theta_{\cal P}} d\pmb{\lambda}=\int_{\tilde{S}_{\cal P,\si}} e^{i\eps^{-2}\theta_{\cal P}\lambda_\ell} e^{i\eps^{-2} \tilde{\Theta}_{\cal P}} d\pmb{\lambda}\to0.
 \]
 Since the above integral is uniformly bounded by ${\sf m}(S_{\mathcal{P},\sigma})$, an application of the dominated convergence theorem proves $\mathcal{I}_\eps(\mathcal{P},\sigma)\to0$.

Now let us enter the details of the discussion. Obviously,
$\lam_1\in \{s_{1,1}, s_{2,1}, s'_{1,1}, s_{2,1}'\}$. Suppose that
$\lam_1= s_{1,1}$ (so $w_1=p_{1,1}$), we claim that \eqref{e.4223} holds for $\lam_1'\not=
s_{1,1}'$. Indeed, if the latter holds, 
then either ${\sf m}(\cal S_{{\cal P},\si})=0$, or   we have $\lam_1'\in
\{s_{2,1}, s_{2,1}'\}$. Assume that $\lam_1'=s_{2,1}'$. As we shall
see from the argument below, the case $\lam_1'=s_{2,1}$ can be treated similarly.

%\gucomment{i don't quite get this part. For a given partition $\mathcal{P}$ and permutation $\sigma$, can we just define the variables $\lambda_\ell$ so that $(\lambda_1,\ldots,\lambda_{K/2})=(r_{\sigma(1)},\ldots,r_{\sigma(K/2)})$? In other words, we give those $w_1-$variables different names, depending on each permutation. In this way, we do not need to introduce the definition of $\tilde{\cal S}_{\mathcal{P},\sigma}$}

%\textcolor{red}{\bf DOTAD}

%First, recall that $\Theta=\textstyle\sum_{j=1}^{2}\Theta_{n_j}(\xi_j,\mathbf{p}_j,\mathbf{u}_j)-\textstyle\sum_{\ell=1}^{2}\Theta_{m_\ell}(\eta_\ell,\mathbf{q}_\ell,\mathbf{v}_\ell)$ with the phase factor defined in \eqref{e.defphase}. 
% The idea is that if $\sigma\notin S_{\mathrm{per}}$, then after integrating out the Dirac functions 
% \[
% \prod_{(\lambda_1,\lambda_2)\in \mathcal{P}_\sigma}\delta(\lambda_1-\lambda_2)\delta(w_1+w_2),
% \] the phase factor $\Theta\neq 0$ so we have a large phase which makes the integral in time small.

% The largest element $r_1$ must be in the set $\{u_{1,1},u_{2,1},v_{1,1},v_{2,1}\}$. Without loss of generality, assume $r_1=u_{1,1}$. There are two cases: (i)  $u_{1,1}\nsim v_{1,1}$; (ii) $u_{1,1}\sim v_{1,1}$. In the first case, we know that $u_{1,1}$ must be paired with $u_{2,1}$ or $v_{2,1}$, since $\mathcal{P}_\sigma$ is a ladder pairing.

Since $s_{1,1}$ is paired with $s_{2,1}'$, we have $p_{1,1}$ paired with $p_{2,1}'$, so the associated phase factor $\Theta_{\cal P}$ equals 
\begin{align*}
\Theta_{\cal P}=&\frac{s_{1,1}}{2}(|\xi_1|^2-|\xi_1-p_{1,1}|^2)-\frac{s_{2,1}'}{2}(|\xi_2'|^2-|\xi_2'-p_{2,1}'|^2)+\tilde{\Theta}_{\cal P}\\
=&\frac{ s_{1,1}}{ 2} \Big[(|\xi_1|^2-|\xi_1-p_{1,1}|^2)-(|\xi_2'|^2-|\xi_2'-p_{1,1}|^2)\Big]+\tilde{\Theta}_{\cal P}\\
=&(\xi_1-\xi_2')\cdot p_{1,1} s_{1,1}+\tilde{\Theta}_{\cal P},
\end{align*}
where $\tilde{\Theta}_{\cal P}$ involves the temporal variables
$\lam_2,\ldots,\lam_{K/2}$ and the momentum  variables
$w_1,\ldots,w_{K/2}$ that we do not track. By the definition of $\xi_1,\xi_2'$ in \eqref{012005-20}, we have $\xi_1-\xi_2'=p'-p$, which yields
\begin{equation}
\begin{aligned}
\label{IPS1}
\mathcal{I}_\eps(\mathcal{P},\sigma)=\int_{\R^{Kd/2}}\int_{ S_{\cal P,\si}}&\left\{\prod_{\ell=1}^{K/2}\frac{\hat{R}(w_\ell)}{(2\pi)^{d}}\right\} \exp\left\{i\eps^{-2}(p'-p)\cdot
   w_1 \lam_1\right\} \\
   &\times e^{i\eps^{-2}\tilde \Theta_{\cal P}}
  \Phi_{\cal P} \, d \pmb{\lam} d\mathbf{w}.
\end{aligned}
\end{equation}

Let 
$\tilde{S}_{\cal P,\si}^2:=\pi\big( \tilde{S}_{\cal P,\si}\big)$, where $\pi:\bbR^{K/2}\to \bbR^{K/2 -1}$ is the
coordinate projection:
$$
\pi(\lam_1,
\lam_2,\ldots,\lam_{K/2}):=(\lam_2,\ldots,\lam_{n}),\quad (\lam_1,
\lam_2,\ldots,\lam_{K/2})\in \bbR^{K/2}.
$$
Note that
$$
\tilde{S}_{\cal P,\si}=\left\{(\lam_1,\ldots,\lam_{K/2}):\,
  (\lam_2,\ldots,\lam_{K/2})\in  \tilde{S}_{\cal P,\si}^2,\,
  \lam_1\in(\lam_2,t_1\wedge t_2')\right\}.
$$

Then, we can rewrite \eqref{IPS1} as 
\begin{align}
\label{IPS2}
&
\mathcal{I}_\eps(\mathcal{P},\sigma)=\int_{\R^{Kd/2}}\Phi_{\cal P}  \left\{\prod_{\ell=1}^{K/2}\frac{\hat{R}(w_\ell)}{(2\pi)^{d}}\right\} d\mathbf{w} \int_{  S^2_{\cal P,\si}}e^{i\eps^{-2}\tilde \Theta_{\cal P}} d{\lam}_2\ldots d\lam_{K/2}\notag\\
&
\times \int_{\lam_2}^{t_1\wedge t_2'}\exp\left\{i\eps^{-2}(p'-p)\cdot
  w_1 \lam_1\right\} 
  \, d {\lam}_1 .
\end{align}
Note that both $\Phi_{\mathcal{P}}$ and the integral 
\[
 \int_{  S^2_{\cal P,\si}}e^{i\eps^{-2}\tilde \Theta_{\cal P}} d{\lam}_2\ldots d\lam_{K/2}
 \]
 are bounded. Then we can argue that
$\lim_{\eps\to0}\mathcal{I}_\eps(\mathcal{P},\sigma)=0$, by the
dominated convergence theorem and the following  fact
\[
\lim_{\eps\to0}\sup_{a,b\in[0, t_1\wedge t_2']}\left|\int_a^b \exp\left\{i\eps^{-2}(p'-p)\cdot
  w_1 \lam_1\right\} d\lam_1\right|=0
\]
that holds for all  $w_1$ such that $(p'-p)\cdot
w_1\not=0$. This is where we rely on the assumption of $p'\neq p$.
Hence, in order for \eqref{e.4223n} to be true we need to pair
$s_{1,1}$ with $s'_{1,1}$.

Next, we note that $\lam_2\in
\{s_{1,2},s'_{1,2},s_{2,1},s'_{2,1}\}$. We argue, in exactly the same
fashion as in the case of $\lam_1$ that if  $\lam_2=s_{i,j}$ and
\eqref{e.4223n} is in force, then 
\begin{equation}
\label{012205-201}
\mbox{$s_{i,j}$ is paired with
$s_{i,j}'$}.
\end{equation}
 The above argument holds also for other choices of $\lam_2$, which 
 propagates down to all other $\lam_\ell$. Finally, we conclude that
\eqref{e.4223n} forces condition \eqref{012205-201} for all
$i=1,2$ and $j=1,\ldots,n_i$. But the latter means that ${\cal
  P}$ is the ladder diagram, which stands in a contradiction to
our assumption \eqref{022205-20}. This ends the proof of
\eqref{e.4223}, which in turn finishes the proof of Lemma \ref{l.com4n2}.
%As we have already observed this also concludes the proof of Lemma
%\ref{l.con2ndmm}.
%\qed 
\end{proof}

%\textcolor{red}{\bf DOTAD}

\subsubsection{Convergence of $\mathscr{Q}_{\eps,j,k}(t)$}
Recall that $\mathscr{Q}_{\eps,j,k}(t)$ is given by \eqref{sqjk}.
% \[
% \begin{aligned}
% \quad \mathscr{Q}_{\eps,j,k}(t) =&\la M_{\eps,\eta_j},M_{\eps,\eta_k}\ra(t)\\
% =&\int_0^t\int_{\R^d} \frac{\hat{R}(p)}{(2\pi)^d} e^{i(\eta_j-\eta_k-\frac{p}{\eps^2})\cdot p s} \psi_\eps(s,\xi+\eps^2 \eta_j-p)\psi_\eps(s,\xi+\eps^2\eta_k+p)dpds.
% \end{aligned}
% \]
The limit in \eqref{e.262a} is a consequence of the following lemma.
\begin{lemma}
For any $t>0$, $j,k=1,\ldots,N$, we have
\begin{equation}
\label{012205-20}
\begin{aligned}
\lim_{\eps\to0}\E[|\mathscr{Q}_{\eps,j,k}(t)|^2]=0.
\end{aligned}
\end{equation}
\end{lemma}

\begin{proof}
Recall that
\begin{equation*}
\begin{aligned}
\mathscr{Q}_{\eps,j,k}(t) =&\la {\cal M}_{\eps}(\cdot,\eta_j), {\cal M}_{\eps}(\cdot,\eta_k)\ra(t)\\
=&-\int_0^t\int_{\R^d} \frac{\hat{R}(p)}{(2\pi)^d} \exp\left\{i{\left(\eta_j-\eta_k-\frac{p}{\eps^2}\right)}\cdot p s\right\} g_\eps^{j,k}(s,\xi,p) dpds,
\end{aligned}
\end{equation*}
with $g_\eps^{j,k}$ defined in \eqref{gjk}. We have
\begin{equation}\label{e.291}
\E[|\mathscr{Q}_{\eps,j,k}(t)|^2]=\int_{[0,t]^2}\int_{\R^{2d}}\frac{\hat{R}(p)\hat{R}(p')}{(2\pi)^{2d}}e^{i\eps^{-2}\theta_\eps(s,s',p,p')}\tilde{\M}_{4,\eps}(s,s',p,p')dpdp'dsds',
\end{equation}
with
\[
\theta_\eps(s,s',p,p')=\eps^2(\eta_j-\eta_k)\cdot (ps-p's')-(|p|^2s-|p'|^2s')
\]
and
\[
\begin{aligned}
\tilde{\M}_{4,\eps}(s,s',p,p')=\E[&\hat\psi_\eps(s,\xi+\eps^2 \eta_j-p) \hat\psi_\eps(s,\xi+\eps^2\eta_k+p)\\
\times&\hat\psi_\eps^*(s',\xi+\eps^2 \eta_j-p') \hat\psi_\eps^*(s',\xi+\eps^2\eta_k+p')].
\end{aligned}
\]
By Corollary~\ref{c.bdpsieps}, we have
\[
\sup_{s,s'\in[0,t],p,p'\in\R^d}|\tilde{\M}_{4,\eps}(s,s',p,p')| \leq C(t).
\] 
%thus, we can estimate $\E[|\mathscr{Q}_{\eps,j,k}(t)|^2]$ as 
%\[
%\E[|\mathscr{Q}_{\eps,j,k}(t)|^2] \les\int_{\R^{2d}} \hat{R}(p)\hat{R}(p')\sup_{r_1\in[0,t]}\left|\int_{r_1}^te^{2i\eps^{-2}[\eps^2(\eta_j-\eta_k)\cdot p-|p|^2]s} dsdpdp'
%\]
The integral in $dsds'$ involves  a large phase factor
$\eps^{-2}\theta_\eps$, which should be  contrasted with our
calculation in case of
$Q_{\eps,j,k}(t)$, see \eqref{qjk}, where this situation happened only
for the non-ladder pairings. The presence of such a factor explains why  the expression in the
left hand side of  \eqref{e.291}  vanishes, as $\eps\to0$. 
To prove this fact rigorously, we write $\tilde{\M}_{4,\eps}(s,s',p,p')$ in terms of
the diagram expansion and proceed, as in Section \ref{sec3.4.2},
integrating first the largest temporal variables in $s,s'$,
analogous to what has been done in \eqref{IPS2}. As the argument is very similar, we do not provide all details but only the sketch.

First, we have
$$
\tilde{\M}_{4,\eps}(s,s',p,p')=\sum_{\bfn, \bfn'}\tilde{\cN}_\eps(\bfn,\bfn',\mathbf{t},\mathbf{t}',\pmb{\xi}, \pmb{\xi}') ,
$$
where, cf \eqref{021805-20b}, $\mathbf{t},\mathbf{t}',\pmb{\xi},
\pmb{\xi}'$ are given by 
\begin{equation}
\label{012005-20c}
\begin{aligned}
&t_1=t_2=s,\quad t_1'=t_2'=s', \\
&\xi_1=\xi+\eps^2\eta_j-p, \quad \xi_2=\xi+\eps^2\eta_k+p, \quad \xi'_1=\xi+\eps^2\eta_j-p',\quad \xi'_2=\xi+\eps^2\eta_k+p',
\end{aligned}
\end{equation}
and
\begin{align}
\label{021805-20c}
&\tilde{\cN}_\eps(\bfn,\bfn',\mathbf{t},\mathbf{t}',\pmb{\xi},
  \pmb{\xi}') :=
  \sum_{\mathcal{P}}\sum_{\sigma}\int_{\R^{Kd/2}}\int_{S_{\cal
  P,\si}}\left\{\prod_{\ell=1}^{K/2}\frac{\hat{R}(w_\ell)}{(2\pi)^{d}}\right\}e^{i\eps^{-2}\hat{\Theta}_{\cal P}}
  \hat{\Phi}_{\cal P} \, d \pmb{\lam}d\mathbf{w} ,
\end{align} 
where $\hat{\Theta}_{\cal P}$ is some, appropriately defined phase factor, 
involving only the variables $\pmb{\lam}$ and $\mathbf{w}$, and
$\hat{\Phi}_{\cal P}$ is the expression corresponding to the product
of the initial data. Here
we use the same  notation for variables, pairing ${\cal P}$,
permutation $\si$ and the domain of integration as in Section   \ref{sec3.4.2}.
We can write then
\[
\lim_{\eps\to0}\E[|\mathscr{Q}_{\eps,j,k}(t)|^2]=\sum_{\bfn,
  \bfn'}\sum_{\mathcal{P},\si}\lim_{\eps\to0}\mathcal{J}_\eps(\mathcal{P},\sigma),
\]
where
\begin{align*}
\mathcal{J}_\eps(\mathcal{P},\sigma):=\int_{{\cal T}_{{\cal
  P},\si}} dsds'  d \pmb{\lam}&\int_{\R^{2d+{Kd/2}}}d\mathbf{w} dpdp'
  \\
&
\times e^{i\eps^{-2}\theta_\eps(s,s',p,p')}\frac{\hat{R}(p)\hat{R}(p')}{(2\pi)^{2d}}
\prod_{\ell=1}^{K/2}\frac{\hat{R}(w_\ell)}{(2\pi)^{d}}e^{i\eps^{-2}\hat{\Theta}_{\cal P}}
  \hat{\Phi}_{\cal P} 
\end{align*}
and 
$$
{\cal T}_{{\cal
  P},\si}:=\left\{(s,s', \pmb{\lam}):\, (s,s')\in[0,t]^2,\, \pmb{\lam}\in S_{{\cal
  P},\si}\right\}.
$$
We emphasize that $S_{{\cal
  P},\si}$ in fact depends on $s$ and $s'$, through the dependence of
$\tilde{\cN}_\eps$ on $\mathbf{t},\mathbf{t}'$. Without loss of
generality, consider the region of $s>s'$. Given the partition
$\mathcal{P}$ and the permutation $\sigma$, the largest
$\pmb{\lambda}$ variable is $\lambda_{\sigma(1)}$, thus for fixed
$s'$ and $\pmb{\lambda}$, the domain of integration for $s$ is
$[\lambda_{\sigma(1)}\vee s',t]$. Using the fact that 
  \[
  \sup_{s',\lambda_{\sigma(1)}\in [0,t]} \left|\int_{\lambda_{\sigma(1)}\vee s'}^t \exp\left\{i (\eta_j-\eta_k)\cdot ps-\eps^{-2}
  |p|^2s\right\}ds\right|\to0, \quad\quad \mbox{ for $p\neq 0$},
  \]
  and applying dominated convergence theorem, we conclude the proof of $\mathcal{J}_\eps(\mathcal{P},\sigma)\to0$ as $\eps\to0$.  
%  Let $\tilde{\cal T}_{{\cal
%  P},\si}$ be the projection of ${\cal T}_{{\cal
%  P},\si}$ onto the $\pmb{\lam}$ coordinates. Then, 
%\begin{align*}
%&\mathcal{J}_\eps(\mathcal{P},\sigma)=\int_{\tilde{\cal T}_{{\cal
%  P},\si}}  d \pmb{\lam}\int_{\R^{2d+{Kd/2}}}\frac{\hat{R}(p)\hat{R}(p')}{(2\pi)^{2d}}
%\prod_{\ell=1}^{K/2}\frac{\hat{R}(w_\ell)}{(2\pi)^{d}}d\mathbf{w} dpdp'
%  \\
%&
%\times e^{i\eps^{-2}\hat{\Theta}_{\cal P}}
%  \hat{\Phi}_{\cal P} \times \left(\int_{[\lam_1,t]^2}\exp\left\{i (\eta_j-\eta_k)\cdot (ps-p's')-\eps^{-2}
%  (|p|^2s-|p'|^2s')\right\}ds ds'\right).
%\end{align*}
%It is easy to verify that for any $(p,p')$ and $(s,s')$ such that
%$|p|^2s-|p'|^2s'\neq 0$ and $0<a<b$ we have
%\[
%\begin{aligned}
%&\lim_{\eps\to0}\int_{[a,b]^2}\exp\left\{i (\eta_j-\eta_k)\cdot (ps-p's')-\eps^{-2}
%  (|p|^2s-|p'|^2s')\right\}ds ds'=0.
%\end{aligned}
%\]
%Therefore, by virtue of the   dominated convergence theorem, $\lim_{\eps\to0}\mathcal{J}_\eps(\mathcal{P},\sigma)=0$. 
%Thus, the proof of \eqref{012205-20} is complete.
\end{proof}

\section{Proofs of the results from Section \ref{sec011109-20}}

\label{s.c}

We show only how to prove Proposition~\ref{c.c1a}. The other results
from Section  \ref{sec011109-20} can be argued similarly.

\subsection{Proof of  (\ref{021109-20ab})}

Thanks to the moment estimate proved in
Corollary \ref{c.bdpsieps} we conclude that 
\begin{equation}
\label{c.bdpsieps-1}
\sup_{t\in[0,T],\xi,\eta\in\R^d,\eps\in(0,1]}\E[|X_\xi^{\eps}(t,\eta)|^{2n}]<+\infty,
\end{equation}
for any integer $n\ge0$ and $T>0$.
Therefore, it suffices only to
show (\ref{021109-20ab}) for $J(x,\xi)$ whose Fourier transform in
the $x-$variable - $\hat J(\eta,\xi)$ - is compactly supported.

Fix $\xi\in\bbR^d$ and $t\ge0$. Suppose also that $\eps_n\to0+$. Thanks to the Skorokhod embedding
theorem, see e.g.  Theorem I.6.7 of \cite{billingsley} and Theorem \ref{t.mainth}, we can assume
that there exist
 a sequence of the fields  $\left(\tilde
   X_\xi^{\eps_n}(t,\cdot)\right)_{n\ge1}$ and a field $\tilde
   X_\xi(t,\cdot)$  such that
\begin{itemize}
\item[i)] the law of  $\tilde
   X_\xi^{\eps_n}(t,\cdot)$ coincides with that of $
   X_\xi^{\eps_n}(t,\cdot)$ for each $n\ge1$. Likewise the laws of  $\tilde
   X_\xi(t,\cdot)$  and that of  $
   X_\xi(t,\cdot)$  are equal,
\item[ii)] 
$\tilde
   X_\xi^{\eps_n}(t,\cdot)$ converge  to $\tilde X_\xi(t,\cdot)$, as $n\to+\infty$, uniformly on compact
   subsets of $\bbR^d$, for a.s. realization of the fields.
\end{itemize}
In consequence, the law of $\int_{\bbR^{d}}{\cal
  X}_\eps(t,x,\xi)J^*(x,\xi)dx$
coincides with that  of  
\begin{equation}
\label{021109-20abc}
\frac{1}{(2\pi)^d}\int_{\bbR^{d}}\tilde
   X_\xi^{\eps_n}(t,\eta) \hat J^*(\eta,\xi)d\eta.
\end{equation}
It follows from ii) that the expressions in \eqref{021109-20abc}
converge a.s. (thus also in law), as $n\to+\infty$, to
\begin{equation}
\label{021109-20abc}
\frac{1}{(2\pi)^d}\int_{\bbR^{d}}\tilde
   X_\xi(t,\eta) \hat J^*(\eta,\xi)d\eta.
\end{equation}
Thus,  (\ref{021109-20ab}) follows.

\subsection{Proof of (\ref{021109-20a})}

Thanks to Theorem \ref{t.mainth} and estimate \eqref{c.bdpsieps-1} we
infer that
\begin{equation}
\label{021109-20abd}
\lim_{\eps\to0}\frac{1}{(2\pi)^d}\bbE\left[\int_{\bbR^{2d}}
   X_\xi^{\eps}(t,\eta) \hat J^*(\eta,\xi)d\eta d\xi\right]=\frac{1}{(2\pi)^d}\int_{\bbR^{2d}}\bbE
   X_\xi(t,\eta) \hat J^*(\eta,\xi)d\eta d\xi.
\end{equation}
The expression in the right hand side equals to the right hand side of
\eqref{021109-20a}, by virtue of \eqref{e.limiteq}.

Using Remark \ref{r.multi} we can also easily conclude that
\begin{equation}
\label{021109-20abd1}
\lim_{\eps\to0}\bbE\left[\int_{\bbR^{2d}}
   X_\xi^{\eps}(t,\eta) \hat J^*(\eta,\xi)d\eta d\xi\right]^2=\left[\int_{\bbR^{2d}}\bbE   X_\xi(t,\eta) \hat J^*(\eta,\xi)d\eta d\xi\right]^2,
\end{equation}
thus (\ref{021109-20a}) follows.

\end{document}